\documentclass{article}
\usepackage{proof,latexsym,amssymb,graphics}
\usepackage[all]{xy}
\newtheorem{definition}{Definition}[section]

\newtheorem{lemma}[definition]{Lemma}
\newtheorem{corollary}[definition]{Corollary}
\newtheorem{proposition}[definition]{Proposition}

\newtheorem{remark}[definition]{Remark}
\newenvironment{proof}{\vspace{3pt}\textsc{Proof:}\quad }
                       {\hfill \hbox{q.e.d.}\vspace{3pt}}
\def\A{{\mathcal A}} 
\def\LL{{\mathcal L}}
\def\P{{\mathcal P}} 
\def\S{{\mathcal S}} \def\T{{\mathcal T}}

\def\sub{\subseteq }

\renewcommand{\land}{\mathrel\&}
\def\releps{\mathrel\epsilon}
\def\cov{\lhd}
\def\fregiu{\mathord\downarrow}
\def\fregiubi{\mathrel\downarrow}
\renewcommand{\land}{\mathrel\&}

\def\to{\rightarrow}

\global\oddsidemargin=5mm \global\evensidemargin=5mm

\begin{document}
\bibliographystyle{siam}
\title{Convergence in Formal Topology: a unifying notion}
\date{}
\author{Francesco Ciraulo \quad Maria Emilia Maietti \quad Giovanni Sambin\\[8pt]
\normalsize Dipartimento di Matematica, University of  Padova\\
%\normalsize University of  Padova\\
\normalsize e-mail: \{ciraulo,maietti,sambin\}@math.unipd.it}

\maketitle

\begin{abstract}
\noindent Several variations on the definition of a Formal Topology
exist in the literature. They differ on how they express
\emph{convergence}, the formal property corresponding to the fact
that open subsets are closed under finite intersections. We
introduce a general notion of convergence of which any previous
definition is a special case. This leads to a predicative
presentation and inductive generation of locales (\emph{formal
covers}), commutative quantales (\emph{convergent covers}) and
suplattices (\emph{basic covers}) in a uniform way. Thanks to our
abstract treatment of convergence, we are able to specify
categorically the precise sense according to which our inductively
generated structures are free, thus refining Johnstone's coverage
theorem.

We also obtain a natural and predicative version of a fundamental
result by Joyal and Tierney: convergent covers (commutative
quantales) correspond to commutative co-semigroups over the category
of basic covers (suplattices).
\end{abstract}


\section{Introduction}

This paper aims to contribute to the development of constructive
topology. By constructive topology we mean topology developed in a
predicative and intuitionistic foundation. In order to avoid
impredicative definitions the foundation must distinguish sets from
collections; a typical example of a collection is given by all
subsets of a set.  The
 usual axiom of separation is then restricted
to formulas that do not contain quantifications over collections.
See \cite{MS05} and \cite{M09} for a formal system and for further
explanations about such a foundation.

It is commonly accepted that, in order to develop topology
constructively, the pointfree approach of locale theory~\cite{stone spaces}
 is the most convenient~\cite{PML70,Fourman&Grayson, johnstone-points,somepoints}.   
The predicative development of locale theory, 
 started by Per Martin-L{\"o}f and the third author in
\cite{IFS},  is now known as Formal Topology.  
To define  a locale predicatively, one needs 
 a base of opens that is a set, while the whole locale is only a 
 collection~\cite{curi}. 
 The original notion of formal
topology, proposed in \cite{IFS}, corresponds to that of  open (or overt)
locale~\cite{galoisexten,negridomains}. Here we call {\it formal cover} 
the generalization of formal topology corresponding to a locale;
it is obtained by
simply dropping the so-called {\it positivity predicate} in \cite{IFS}.
Because of the presence of bases,  morphisms between formal covers/formal 
topologies are suitable relations.
So they acquire a direct intuitive interpretation both in
the direction of locales and in the opposite direction, namely that
of frames~\cite{stone spaces}.

Since its introduction in \cite{IFS}, the notion of formal topology
has been presented in several different ways. One of the motivations was
that of including relevant examples in a direct way, without 
artificial tricks. For instance, the original version in \cite{IFS}, or that in \cite{somepoints},
works well for Stone spaces \cite{saranegri},  Scott domains,
Zariski topologies,..., while  the variant in \cite{ftopposets}, \cite{cssv} is more
suitable for Baire spaces,  algebraic domains, Kripke models and
discrete topologies.

No variant superseded the others; actually, while in the general
case they are all equivalent, they are no longer equivalent in the
unary or finitary case. For example, one version of unary formal
topologies represents Scott domains~\cite{SVV}, while another
algebraic domains~\cite{sambin-domains as formal topologies}.

All presentations of formal topology in the literature differ only
in their way of expressing closure of opens under finite
intersections, which in a pointfree approach appears as
distributivity of finite meets over arbitrary joins in the lattice
of opens. This property is here called {\it convergence}.
In order to express convergence, in this paper we introduce a
binary operation $\circ$ on subsets of the base with suitable
conditions. We thus achieve a new unifying notion of formal cover/formal topology.
All previous presentations  are obtained as a special case by
imposing some further conditions on the operation $\circ$.

Our definition gives a new predicative presentation of locales.
This is obtained in a modular way (see table page~\pageref{tabella})
starting from suplattices and passing through a new presentation of
quantales. The constructive notion corresponding to suplattices is
called  \emph{basic cover} \cite{somepoints,bp}. It is thought of
as a generalized pointfree topology without convergence. The
category of basic covers is (impredicatively) dual to that of
suplattices and hence it gives a genuine generalization of the
category of locales.

By choosing to work in the direction of locale maps, the category of
basic covers becomes the right setting to prove a predicative
counterpart of Joyal-Tierney's result \cite{galoisexten} stating
that frames
 are  special commutative monoids over the category of
suplattices. We achieve this by introducing the notion of
\emph{convergent cover}, namely a basic cover equipped with a weak form of the operation
$\circ$. The category of convergent covers is
(impredicatively) dual to that of commutative quantales.\footnote{Actually, 
a constructive presentation of quantales is possible also
using the notion of pretopology~\cite{pretopologies}, but that
presentation does not allow to prove  Joyal-Tierney's result in a
direct way.}

By means of the new notions we can prove
 Joyal-Tierney's result in the following dualized form:
convergent covers are commutative co-semigroups in the category of
basic covers.  A predicative proof of this is possible as soon as
the tensor product exists predicatively, a fact which happens in the
inductively generated case.

Just as a locale is  a quantale in which multiplication and meet coincide,
a \emph{formal cover} is a
convergent cover in which the operation $\circ$ corresponds to the
lattice-theoretic meet. The category of formal covers is then
(impredicatively) dual to that of  locales.

The inclusion  of all previous definitions in
our new one leads to a unified, general method of inductive
generation that applies to all variants. In particular, the rules
of inductive generation first given in \cite{cssv}  become a special case of ours. 
In addition, we provide a method for generating also suplattices 
(basic covers) and quantales (convergent
covers) in a modular way.

Thanks to the abstract character of our presentation of convergence,
we are able to give a categorical reading of the inductive
generation of formal covers and convergent covers. We can specify in
what sense these constructions are free by showing that they provide 
object parts of adjoints to suitable forgetful functors. These
results can be read as a refinement of Johnstone's coverage theorem~\cite{stone spaces,vickers}.

All the definitions and results of the present paper work equally
well, with no modification, also when a \emph{positivity relation}
$\ltimes$ is added besides the cover \cite{somepoints,bp}. In fact
the addition of $\ltimes$, whose aim is to give a primitive
pointfree version of closed subsets, does not affect the notion of
convergence.

A general treatment of convergence as the one given here seems to be
 a necessary step towards a purely algebraic development
of constructive topology. The present approach via the operation
$\circ$ on subsets can be easily generalized to the algebraic
framework basing on overlap algebras \cite{regular,bp,CMT}.

\section{Predicative suplattices: the notion of a basic cover} %%%%%%%%%%%%%%%%

The notion of basic cover recalled below can be read as a
predicative topological presentation of a suplattice, or complete
join-semilattice \cite{stone spaces}. The corresponding notion of
morphism makes the category of basic covers {\bf BCov} dual
(impredicatively) to that of suplattices and sup-preserving
maps. This choice for the direction of arrows is justified by
the fact that a suitable subcategory of {\bf BCov} becomes
equivalent to the category of locales; indeed, reaching a predicative
version of locales in the context of basic covers is one of our
aims.

We use the notation $Y\sub X$ to mean that $Y$ is a subset of $X$,
where $X$ can be either a set or a collection. 
A subset in our foundation is defined by a propositional function with
quantifications restricted to sets (see  \cite{M09} and \cite{bp} for a more precise explanation).
The collection of all subsets of a set $S$ is denoted by $\P(S)$.

\begin{definition}\label{def. basic cover}
Let $S$ be a set. A \emph{basic cover} on $S$ is a relation
$\cov\sub S\times\P(S)$ between elements and subsets of $S$ that 
satisfies the following rules for every $a,b\in S$ and $U,V\sub S$:
$$
\infer[\;reflexivity\qquad and]{a\cov U}{a\releps U}\qquad
\infer[\;transitivity]{a\cov V}{a\cov U & U\cov V}
$$ 
where
$U\cov V$ $\stackrel{def}{\Longleftrightarrow}$ $(\forall\, b\releps
U)$ $(b\cov V)$.
\end{definition}

Although this notion is pretty general, $S$ is often interpreted
as a set of (names of) open subsets of a topology, typically a base.
Then $a\cov U$ is read: ``the open subset (whose name is) $a$ is
contained in the union of those belonging to $U$''. For the sake of
notation, we shall often confuse elements with singletons; for
instance, we shall write $a\cov b$ instead of $a\cov\{b\}$, for
$a,b\in S$. Moreover, we shall use the term ``basic cover'' also for
the pair $(S,\cov)$ itself.

For every basic cover $(S,\cov)$ and every subset $U\sub S$, we put:
\begin{equation}
\A U\stackrel{def}{=}\{a\in S\ |\ a\cov U\}.
\end{equation}
This defines a \emph{saturation} (or closure operator) on $\P(S)$, that is, a map 
$\A$ $:$ $\P(S)$ $\longrightarrow$ $\P(S)$ which is monotone (with respect to
inclusion), idempotent and expansive (that is, $U\sub \A U$ for all $U\sub S$). 
Vice versa, if $\A$ is a
saturation on $\P(S)$, then  $a\cov U$
$\stackrel{def}{\Leftrightarrow}$ $a\releps\A U$ defines a basic cover on $S$. The
correspondence between basic covers on $S$ and saturations on
$\P(S)$ is a bijection (see \cite{bs} for details).

\begin{definition}
For every basic cover $(S,\cov)$ and every subset $U\sub S$,  we say
that $U$ is a \emph{formal open}, or \emph{saturated}, if $U=\A U$.  We
write $U=_\A V$ for $\A U=\A V$.
The collection of all formal open subsets is written $Sat(\A)$. 
\end{definition} 
Since $\A$ is idempotent, $Sat(\A)$ can be described also
as the collection of all subsets of the form $\A U$, for $U\sub S$.
Moreover, it is easy to see that $Sat(\A)$ can be identified with
the quotient of $\P(S)$ modulo the equivalence relation $=_\A$.

It is well known that the collection of all fixed points of a
saturation $\A$ can be given the structure of a suplattice. 
Joins are defined by:
\begin{equation}\label{def join SatA}
{\bigvee_{i\in I}}^\A\A W_i\ \stackrel{def}{=}\ \A\bigcup_{i\in I}\A W_i\ =\
\A\bigcup_{i\in I}W_i\,.
\end{equation} 
It is easy to see that the second equality holds; it says that  $=_\A$
is respected by unions.

As a suplattice, $Sat(\A)$ is generated by the
set-indexed family $\{\A a\ |\ a\in S\}$. Vice versa, if $\LL$ is a
suplattice which admits a set $S\sub\LL$ of
generators,\footnote{This requirement is not trivial since,
predicatively, suplattices are not sets (see \cite{curi}).} then the
structure $(S,\cov)$, where $a\cov U$ if $a\leq\bigvee U$, is a
basic cover whose corresponding $Sat(\A)$ is isomorphic to $\LL$.
Thus, at least impredicatively, every suplattice is of the form
$Sat(\A)$ (see \cite{bs} for details).

Impredicatively, every suplattice has a meet operation too. In the
case of the suplattice $Sat(\A)$, meets exist also predicatively and
are given by
\begin{equation}\label{meet in SatA}
\A U\wedge^\A\A V\ =\ \A(\A U\cap\A V)\ =\ \A U\cap\A
V \, .
\end{equation}

\subsection{Morphisms between basic covers}

When $r$ is a binary relation between two
sets $S$ and $T$, as in \cite{bp}  we define an operator $r^-$ $:$ $\P(T)$
$\longrightarrow$ $\P(S)$ by putting
\begin{equation}
a\releps r^-V\Longleftrightarrow(\exists b\releps V)(a\,r\,b)
\end{equation}
for every $a\in S$ and $V\sub T$.

\begin{definition}\label{def.basic cover map}
Let $\S =(S,\cov_\S)$ and $\T =(T,\cov_\T)$ be two basic covers. A
relation $r$ between $S$ and $T$ is a $\emph{basic cover map}$, or
it is said to respect covers,  if:
\begin{equation}\label{eq. basic cover map} 
b\cov_\T V\quad\Longrightarrow\quad r^-b\cov_\S r^-V
\end{equation} 
for every $b\in T$ and $V\sub T$.
Two basic cover maps $r_1$ and $r_2$ from $\S$ to $\T$ are declared
\emph{equal} if ${r_1}^- b$ $=_{\A_\S}$ ${r_2}^- b$ for all $b\in
T$.\footnote{Thus, properly speaking, a morphism between two basic
covers is an equivalence class of relations satisfying (\ref{eq.
basic cover map}).}
\end{definition}

It is possible to show (see \cite{bp}) that equation (\ref{eq. basic
cover map}) is exactly what is needed to make the assignment $\A_\T
V$ $\longmapsto$ $\A_\S\, r^-V$ a well-defined sup-preserving map from $Sat(\A_\T)$
to $Sat(\A_\S)$. 
Vice versa, each sup-preserving map $h$ $:$ $Sat(\A_\T)$
$\longrightarrow$ $Sat(\A_\S)$ can be obtained in this way. In fact,
it corresponds to the relation $r$ between $S$ and $T$ defined by
$a\,r\,b$ if $a\releps h(\A_\T b)$. One can see that two relations $r_1$
and $r_2$ are equal as basic cover maps exactly when ${r_1}^- V$
$=_{\A_\S}$ ${r_2}^- V$ for all $V\sub T$, that is, exactly when they
correspond to the same map between suplattices.

\begin{proposition}\label{prop. BCov dual SupLat}
Basic covers and basic cover maps (modulo their equality) form a
category, called {\bf BCov}. Identities are represented by (the
class of) identity relations. Composition is usual composition of
relations.

The category {\bf BCov} is (impredicatively) dual to the category
{\bf SupLat} of suplattices and sup-preserving maps.
\end{proposition}
\begin{proof}
It is straightforward to show that {\bf BCov} is a category. The
previous discussion shows that it is dual to {\bf SupLat} (see
\cite{bs} for more details).
\end{proof}

\subsection{Inductive generation of basic covers}
\label{bcovi}

The authors of \cite{cssv} describe a method for \emph{inductively
generating} basic covers and give a predicative justification for
it. Namely, they construct a basic cover satisfying arbitrary axioms
of the form $a\cov U$. The problem is that simply taking the reflexive
and transitive (in the sense of definition \ref{def. basic cover})
closure of the axioms is not a well-founded procedure from a
predicative point of view. In fact,  accepting transitivity as an
inductive rule requires to consider a collection of assumptions, one
for each subset $U$ in the rule.
So, an impredicative argument is necessary to get a fixed point of the 
operator associated with the inductive clauses. This is confirmed by the fact that
some formal topologies  cannot be inductively generated (see \cite{cssv}).

In \cite{cssv} it is shown how to solve this problem. Given a set
$S$, one needs a set-indexed family of axioms of the form $a\cov U$.
This means that one has a set $I(a)$ for each $a\in S$ and, for each
$a\in S$ and $i\in I(A)$, a subset $C(a,i)\sub S$ with the intended
meaning that $a\cov C(a,i)$ holds. The pair $I,C$ is called an
\emph{axiom-set}.

With every axiom-set $I,C$ one can associate a basic cover, say
$\cov_{I,C}$, such that: 
\begin{itemize}
\item[(i)] $a\cov_{I,C}C(a,i)$ for every $a\in S$
and $i\in I(a)$; 
\item[(ii)] if $\cov'$ is another basic cover such that
$a\cov'C(a,i)$ for all $a\in S$ and $i\in I(A)$, then $a\cov_{I,C}
U$ $\Rightarrow$ $a\cov' U$ for all $a\in S$ and $U\sub S$. 
\end{itemize}
In other
words, $\cov_{I,C}$ is the least cover satisfying the axioms
$a\cov_{I,C}C(a,i)$ for all $a\in S$ and $i\in I(a)$. 
One can show (see~\cite{cssv}) that 
$\cov_{I,C}$ is the unique  relation
between elements and subsets of $S$ which satisfies:
\begin{itemize}
\item[i.]  \enspace $\displaystyle{\frac{a\releps U}{a\cov U}}$ \enspace \emph{reflexivity};
\item[ii.]  \enspace $\displaystyle{\frac{i\in I(a) \quad C(a,i)\cov U}{a\cov U}}$ 
 \enspace \emph{infinity} (transitivity restricted to axioms);
\item[iii.]  \enspace \emph{induction}: for every $P\sub S$, if $P$ satisfies
$$
\frac{b\releps U}{b \releps P} \quad \mbox{ and } \quad 
\frac{i\in I(b) \quad  C(b,i)\sub P}{b \releps P} \quad \mbox{ for all $b\in S$}
$$
then $a\cov U$ implies $a\releps P$;
\end{itemize}
for all $a\in S$ and $U\sub S$. 

As recalled in \cite{cssv}, this kind of
inductive definition can be formalized and justified in a
constructive framework such as Martin-L\"of type theory.
In practice, proving $a\cov U \to a\releps P$ by induction on $a\cov U$ 
means checking that
$a\releps P$ holds in either of the two cases: the assumption $a\releps U$ 
and the inductive hypothesis
$C(a,i)\sub P$ for some $i\in I(a)$.

\begin{definition}\label{prop. ind. gen. of basic covers}
A basic cover $(S,\cov)$ is \emph{inductively generated} if there
exists an axiom-set $I,C$ such that $\cov$ = $\cov_{I,C}$, that is,
$\cov$  satisfies the rules above.

We call {\bf BCov$_i$} the full subcategory of {\bf BCov} whose
objects are inductively generated.
\end{definition}

We recall from \cite{cssv} that the category {\bf BCov$_i$} of inductively 
generated basic cover  is predicatively a {\it proper} subcategory of {\bf BCov}, 
since there are examples
of basic covers (actually formal covers!) that cannot be inductively generated.

When restricting to inductively generated basic covers,  
 a relation  is a basic cover map if it respects the axioms in the following sense:

\begin{lemma}\label{lemma continuity for generated covers}
Let $\S$ = $(S,\cov_{\S})$ and $\T$ = $(T,\cov_{\T})$ be two basic covers and
let $r$ be a relation between $S$ and $T$. If $\T$ is inductively
generated by the axiom-set $I,C$, then $r$ is a basic cover map from
$\S$ to $\T$ if and only if $r^-b\cov_{\S} r^-C(b,i)$ holds for all $b\in
T$ and all $i\in I(b)$.
\end{lemma}
\begin{proof} If $r$ is a basic cover map, then for all $b\in T$
and $i\in I(b)$ one has $r^-b\cov_{\S} r^-C(b,i)$ because $b\cov_{\T} C(b,i)$
holds. 
Vice versa, under the assumption $r^-b\cov_{\S} r^-C(b,i)$  for all $b\in
T$ and all $i\in I(b)$, we prove that $b\cov_{\T} V$ $\Rightarrow$
$r^-b\cov_{\S} r^- V$ for all $b\in T$ and $V\sub T$ by
induction on the generation of $b\cov_{\T} V$. 
We have to consider two cases, depending on
whether $b\cov_{\T} V$ follows by reflexivity or infinity. In the first
case, it must be $b\releps V$; so $r^-b\sub r^- V$ and hence $r^-
b\cov_{\S} r^- V$ by reflexivity of $\cov_{\S}$. In the second case, we have
  $C(b,i)\cov_{\T} V$ for some $i\in I(b)$. By the inductive
hypothesis applied to all elements of $C(b,i)$ we get $r^-C(b,i)\cov_{\S}
r^-V$. This, together with the assumption,
gives $r^-b\cov_{\S} r^- V$ (by transitivity of $\cov_{\S}$).
\end{proof}

Impredicatively, every basic cover $(S,\cov)$ can be generated by
means of an axiom-set $I,C$, where $I(a)$ = $\{U\sub S$ $|$ $a\cov
U\}$ for every $a\in S$ and $C(a,U)$ $=$ $U$ for $U\in I(A)$. So
{\bf BCov$_i$} and {\bf BCov} coincide impredicatively.

\begin{proposition}\label{tensor}
{\bf BCov$_i$} is a symmetric monoidal category.\footnote{For the
definition of a monoidal category see \cite{MacLane} p. 161;
``symmetric'' is defined on page 184.}
\end{proposition}
\begin{proof}
Let $\S$ and $\T$ be two basic covers inductively generated by the
axiom-sets $I,C$ and $J,D$, respectively. The tensor $\S\otimes \T$
is the basic cover $(S\times T,\cov_{\S\otimes\T})$ generated by the
axioms \begin{equation}\label{eq. def. tensor product}
(a,b)\cov_{\S\otimes\T}C(a,i)\times
\{b\}\qquad\textrm{and}\qquad(a,b)\cov_{\S\otimes\T}\{a\}\times
D(b,j)\end{equation} for every $(a,b)\in S\times T$ and every $i\in
I(a)$ and $j\in J(b)$. This is a predicative rendering of the
construction of the the tensor product in \cite{galoisexten} as
presented in \cite{Vickers-Johnstone}. We leave the details showing
that this is indeed an axiom-set.

For every pair $(r_1,r_2)$ of morphisms in {\bf BCov$_i$}, with
$r_1$ : $\S_1$ $\longrightarrow$ $\T_1$ and $r_2$ : $\S_2$
$\longrightarrow$ $\T_2$, the tensor $r_1\otimes r_2$ is by definition the unique
morphism from $\S_1\otimes\S_2$ to $\T_1\otimes\T_2$ such that
$(r_1\otimes r_2)^-(b_1,b_2)$ $=_{\A_{\S_1\otimes\S_2}}$
$({r_1}^-b_1)\times({r_2}^-b_2)$ for all $b_1,b_2$.

The unit of the tensor, say $E$, is given by the cover generated on
the singleton set $\{\ast\}$ by means of no axioms. The associative
isomorphisms, the isomorphisms of the unit and the commutative
isomorphisms
$$\begin{array}{l}
\alpha_{\S_1,\S_2,\S_3}\ :\ \S_1\otimes(\,  \S_2\otimes \S_3\, )\
\longrightarrow\ (\, \S_1\otimes \S_2\, ) \otimes \S_3\\[5pt]
\lambda_{\S}\ :\ E\otimes \S\ \longrightarrow\
\S\qquad\qquad\rho_{\S}\ :\ \S\otimes E\ \longrightarrow\ \S
\\[5pt]
\gamma_{\S_1,\S_2}:\ \S_1\otimes \S_2\ \longrightarrow\ \S_2\otimes
\S_1
\end{array}$$
are defined in the obvious way. For instance, $\gamma_{\S_1,\S_2}$
is the unique basic cover map (up to equality) such that
${\gamma_{\S_1,\S_2}}^-(a_2,a_1)$ $=_{\A_{\S_1\otimes\S_2}}$
$(a_1,a_2)$. 
Similarly, $\alpha$ is defined by
${\alpha_{\S_1,\S_2,\S_3}}^-((a_1,a_2),a_3)$ $=_{\A_{(\S_1\otimes
\S_2)\otimes \S_3}}$ $(a_1,(a_2,a_3))$. It is easy to check that
$\alpha$, $\rho$, $\lambda$ and $\gamma$ are all natural and that
the following coherence conditions (see \cite{MacLane}) are
satisfied:
$$\begin{array}{l}
\alpha_{\S_1\otimes\S_2,\S_3,\S_4}\cdot\alpha_{\S_1,\S_2,\S_3\otimes\S_4}
=(\alpha_{\S_1,\S_2,\S_3}\otimes
id_{\S_4})\cdot\alpha_{\S_1,\S_2\otimes\S_3,\S_4}\cdot(
id_{\S_1}\otimes\alpha_{\S_2,\S_3,\S_4})\ ,\\[5pt]
(\rho_{\S_1}\otimes
id_{\S_2})\cdot\alpha_{\S_1,E,\S_2}=id_{\S_1}\otimes\lambda_{\S_2}
\ ,\qquad \lambda_E=\rho_E\ ,\\[5pt]
\gamma_{\S_2,\S_1}={\gamma_{\S_1,\S_2}}^{-1}\ ,\qquad
\rho_\S=\lambda_\S\cdot\gamma_{\S,E}\ ,\\[5pt]
\alpha_{\S_3,\S_1,\S_2}\cdot\gamma_{\S_1\otimes\S_2,\S_3}\cdot\alpha_{\S_1,\S_2,\S_3}
= (\gamma_{\S_1,\S_3}\otimes
id_{\S_2})\cdot\alpha_{\S_1,\S_3,\S_2}\cdot(id_{\S_1}\otimes\gamma_{\S_2,\S_3})\
.
\end{array}$$
\end{proof}

One can prove (see \cite{MV04}) that the tensor functor on a basic
cover $\S$
$$
\S\otimes(-)\ :\  {\bf BCov_i}\longrightarrow {\bf BCov_i}
$$
has {\it impredicatively} the left adjoint
$$
(-)\rightarrow \S\ :\ {\bf BCov_i}\longrightarrow {\bf BCov_i}\ 
$$
where $\T\rightarrow \S$ is the \emph{impredicative} basic cover
corresponding to the suplattice of all basic cover maps from $\T$ to
$\S$ ordered pointwise:
$$
r\leq s\, \equiv\, (\forall \, U\releps \P(S)) \, \ r^- U\cov s^- U \ .
$$
In other words, $\S\otimes(-)$ as a functor on ${\bf BCov_i}^{op}$ has a
right adjoint. Therefore the tensor on {\bf BCov$_i$}, which is {\it
impredicatively} the opposite of the category {\bf SupLat}, coincides with the
Galois tensor defined by Joyal-Tierney in \cite{galoisexten} (see
also \cite{Vickers-Johnstone}).

The above definition of $\otimes$ benefits from a
topological intuition. Indeed, the following lemma shows that the basic cover of
$\S \otimes \T$ satisfies an analogue of a key property
of the product of two topological spaces.

\begin{lemma}\label{lemma tensor}
Let $\S$ and $\T$ be two inductively generated basic covers. Then
$$
\infer{(a,b)\cov_{\S\otimes\T}U\times V}{a\cov_\S U & b\cov_\T V}
$$ 
holds for all $a\in S$, $b\in T$,
$U\sub S$ and $V\sub T$.
\end{lemma}
\begin{proof}
By double induction on the proofs of $a\cov_\S U$ and $b\cov_\T V$.
We must analyze four different cases: $a\releps U$ and $b\releps V$;
$a\releps U$ and $D(b,j)\cov_\T V$ for some $j\in J(b)$;
$C(a,i)\cov_\S U$ for some $i\in I(a)$ and $b\releps V$;
$C(a,i)\cov_\S U$ for some $i\in I(a)$ and $D(b,j)\cov_\T V$ for
some $j\in J(b)$. All cases are proved similarly. For instance, let
$a\releps U$ and $D(b,j)\cov_\T V$ for some $j\in J(b)$, that is
$b'\cov_\T V$ for all $b'\releps D(b,j)$.
Then by inductive hypothesis we get $(a,b')\cov_{\S\otimes\T}U\times V$ for
all $b'\releps D(b,j)$, that is $\{a\}\times
D(b,j)\cov_{\S\otimes\T}U\times V$. Hence
$(a,b)\cov_{\S\otimes\T}U\times V$ by (\ref{eq. def. tensor product}).
\end{proof}

This lemma can be expressed  as $\A_\S U\times\A_\T
V\cov_{\S\otimes\T}U\times V$ and hence also as the equation
\begin{equation}
\label{eq saturazione prodotto}
\A_\S U\times\A_\T
V=_{\A_{\S\otimes\T}}U\times V
\end{equation}
since $U\times
V\cov_{\S\otimes\T}\A_\S U\times\A_\T V$ holds by reflexivity.

\section{Operations on formal opens} %%%%%%%%%%%%%%%%%%%%%

A locale is a suplattice in which binary meets distribute over
arbitrary joins. Since our aims include the inductive generation of
locales, we wish to modify the inductive generation of a basic cover $\A$
so that the resulting lattice $Sat(\A)$ 
(recall that by~(\ref{meet in SatA}) it always has  a meet)
satisfies distributivity.
The mere requirement of distributivity of $Sat(\A)$ says nothing on
how to obtain it when $\cov$ is generated inductively. As we will
see, however, it is possible to impose distributivity by adding an
extra primitive operation $\circ$ on subsets of the base $S$ with
certain suitable properties. In fact, using $\circ$ one can impose
some conditions during the generation process which guarantee that
distributivity holds ``at the end'', when the generation of $\cov$
is ``completed''.  

This method extends in a natural way to the generation of quantales. 
Recall that a quantale is a suplattice with an associative binary operation, 
called multiplication, that is distributive over joins~\cite{rosenthal}.
The idea is to make $Sat(\A)$ a quantale $(Sat(\A),\bigvee^\A,\circ^\A)$
where the multiplication  $\circ^\A$
is induced by an operation $\circ$ on subsets of $S$. In this section we see
what conditions on $\circ$ make $\circ^\A$ well-defined, commutative
and associative in $Sat(\A)$. In the next section we will study the case of
distributivity of $\circ^\A$ and later the special case of locales.

We start by specifying how an operation $\circ^\A$ on $Sat(\A)$ is obtained in
terms of a given operation $\circ$ on $\P(S)$. Our heuristic criterion is to
read an element $\A U$ of $Sat(\A)$ as an ideal object which is
approximated by the concrete subset $U$. This view is suggested by
the case in which $\A$ is inductively generated and thus $\A U$ is
only the ``limit'' of the generation process. Then it is natural to
require that the operation $\circ^\A$ on $Sat(\A)$ is approximated
by the operation $\circ$ on $\P(S)$. Thus we put
\begin{equation}
\label{eq. def. circ^A}
\A U\circ^{\A} \A V\quad=\quad\A(U\circ V)
\end{equation} 
which says that applying $\circ$ to approximations
$U$ of $\A U$ and $V$ of $\A V$ produces an approximation $U\circ V$
of $\A U\circ^\A \A V$. This equation is our starting point to find
the right conditions on the operation $\circ$. First of all, in
order to read equation (\ref{eq. def. circ^A}) as the definition of
$\circ^\A$, we must understand what conditions on $\circ$ make
$\circ^\A$ well-defined.

\begin{proposition}\label{primi equivalenti di stability}
For every basic cover $(S,\cov)$ and every binary operation $\circ$
on $\P(S)$, the following are equivalent:
\begin{enumerate}
\item $\circ^\A$ as defined in~(\ref{eq. def. circ^A}) is 
a well-defined operation on 
 $Sat(\A)$, that is: $$(\A U=\A U')\ \land\ (\A V=\A V')\
\Longrightarrow\ \A(U\circ V)=\A (U'\circ V')$$ for all
$U,U',V,V'\sub S$ (in other words, $\circ$ respects $=_\A$);
\item $\A(\A U\circ\A V)=\A(U\circ V)$,\ \ for all\ \ $U,V\sub S$.
\end{enumerate}
\end{proposition}
\begin{proof}
(1$\Rightarrow$2) Since $\A$ is idempotent, $\A\A U=\A U$ and $\A\A V=\A V$ hold
and then $\A(\A U\circ\A V)$ = $\A(U\circ V)$, that is, 2.

(2$\Rightarrow$1) Assume $\A U=\A U'$ and $\A V=\A V'$. Then $\A
U\circ\A V$ = $\A U'\circ\A V'$ and hence, a fortiori, $\A(\A
U\circ\A V)$ = $\A(\A U'\circ\A V')$. By 2, this gives $\A(U\circ V)$ =
$\A(U'\circ V')$. So $1$ is proved.
\end{proof}

Assuming~(\ref{eq. def. circ^A}), item $2$ above says that 
\begin{equation}\label{def sbagliata}
\A U\circ^\A\A V\ = \ \A(\A U\circ\A V)
\end{equation}
If one takes this equation as a definition of $\circ^\A$,  one  
immediately obtains a well-defined
operation on $Sat(\A)$  without extra requirements
for $\circ$.
Nevertheless, we have not done like that
since~(\ref{def sbagliata}) does not satisfy our intuition on approximations. 
In fact,  (\ref{def sbagliata}) is of no use when $\A$ is inductively
generated since it produces an approximation of $\A U\circ^\A\A V$
only ``after'' the generation of the ideal objects $\A U$ and $\A V$
is ``completed''.

\begin{lemma}
Let $(S,\cov)$ be a basic cover and let $\circ$ be a binary
operation on $\P(S)$. Then the following are equivalent:
\begin{enumerate}
\item $\matrix{\displaystyle{\frac{U\cov V}{U\circ W \cov V \circ
W} }}$ \enspace and \enspace $\matrix{\displaystyle{\frac{U\cov V}{W\circ U \cov
W\circ V} }}$ \enspace localization; \vspace{7pt}
\item  $\matrix{\displaystyle{
\frac{U_1 \cov V_1 \quad U_2 \cov V_2}{U_1 \circ U_2 \cov V_1 \circ
V_2} }}$ \enspace stability;
\item $\circ^\A$ as defined in~(\ref{eq. def. circ^A}) respects inclusion, that is:
$$(\A U_1\sub\A V_1)\land(\A U_2\sub\A V_2)\Longrightarrow
(\A U_1\circ^\A\A U_2)\sub(\A V_1\circ^\A\A V_2)$$ for all
$U_1,U_2,V_1,V_2\sub S$.
\end{enumerate}
\end{lemma}
\begin{proof}
($1.\Leftrightarrow 2.$) Assume localization. If $U_1\cov V_1$ and
$U_2\cov V_2$, then both $U_1\circ U_2\cov V_1\circ U_2$ and
$V_1\circ U_2\cov V_1\circ V_2$ hold; hence $U_1\circ U_2\cov
V_1\circ V_2$ by transitivity. Vice versa, the rules of localization
are particular cases of stability, since  $W\cov W$.

($2.\Leftrightarrow 3.$) 
 Recall that $U\cov V$
iff $\A U\sub\A V$. So $2$ can be rewritten as $\A U_1\sub\A V_1$
$\land$ $\A U_2\sub\A V_2$ $\Rightarrow$ $\A(U_1\circ U_2)$ $\sub$
$\A(V_1\circ V_2)$
that is 3 by definition of $\circ^\A$.
\end{proof}

Item 3, together with~(\ref{eq. def. circ^A}), implies that
$\circ$ respects $=_{\A}$. So each of the three items above is a
sufficient condition for $\circ^{\A}$ to be well-defined.
This is not surprising since they express  monotonicity of $\circ$ 
with respect to the preorder induced by $\cov$ on $\P(S)$, and
$=_\A$ is the equivalence relation associated with it.

Most properties one can require on $\circ^\A$ are induced in a natural
way by corresponding properties linking $\circ$ with the cover. For
instance, $\circ^\A$ is commutative iff $\A(U\circ V)$ = $\A(V\circ
U)$, that is $U\circ V$ $=_\A$ $V\circ U$. Similarly, $\circ^\A$ is
associative iff $\circ$ is associative modulo $=_\A$, that is
$(U\circ V)\circ W$ $=_\A$ $U\circ (V\circ W)$, for all $U,V,W\sub
S$. In this paper, for simplicity's sake, we shall always assume
$\circ^\A$ to be associative and commutative. 

For future reference,
it is convenient to fix a name for a basic cover with an operation
$\circ$ such that $\circ^\A$ is well-defined, monotone, associative
and commutative. Thanks to the previous lemma, the definition can be
reduced to the following form.

\begin{definition}\label{def. BCovwith operation}
We say that $(S,\cov,\circ)$ is a \emph{basic cover with operation}
if $\S=(S,\cov)$ is a basic cover and $\circ$ is a binary operation
on $\P(S)$ which satisfies:
\begin{enumerate}
\vspace{10pt}
\item $\matrix{\displaystyle{
\frac{U_1 \cov V_1\quad U_2 \cov V_2}{U_1\circ U_2 \cov V_1\circ
V_2} }}$ \enspace stability; or, equivalently, localization:
$\matrix{\displaystyle{ \frac{U\cov V}{U\circ W\cov V\circ W} }}$;
\vspace{7pt}
\item $(U\circ V)\circ W$ $\cov$ $U\circ(V\circ W)$ \enspace associativity with respect to
$=_\A$;
\item $U\circ V$ $\cov$ $V\circ U$ \enspace commutativity with respect to
$=_\A$.
\end{enumerate} In this case, the equation $\A U\circ^\A\A V$
$\stackrel{def}{=}$ $\A(U\circ V)$ defines a monotone, associative
and commutative operation on the suplattice $Sat(\A)$.
\end{definition}

Thanks to stability, $\circ^\A$ is a map $Sat(\A)\times
Sat(\A)\rightarrow Sat(\A)$ in the category of partial orders, where
$\times$ is the cartesian product. Items $2$ and $3$ in the
previous definition make $(Sat(\A),\circ^\A)$ a commutative
semigroup in the category of partial orders.

Basic covers with operation provide a
ground framework  for studying the concepts we are
mainly interested in, namely (commutative) quantales and locales.
First, we shall show how to obtain presentations of (commutative)
quantales.

\section{Presenting commutative quantales: convergent covers}
\label{section convergent covers}

>From now on we assume $\S$ = $(S,\cov,\circ)$ to be a basic cover with operation
(in the sense of definition \ref{def. BCovwith operation}). 
In this section we are going to study the case in which the
corresponding 
structure $(Sat(\A),\bigvee^\A,\circ^\A)$ is a commutative quantale
\cite{mulvey,rosenthal}, that is,  when multiplication $\circ^\A$ distributes over arbitrary joins.
This we call a convergent cover. Together with a suitable notion of
morphism, one gets a category which is dual to the category {\bf cQu} of
commutative quantales.  

\begin{definition}
A basic cover with operation $(S,\cov, \circ)$ is called a {\em
convergent cover} if $\circ^\A$ distributes over $\bigvee^\A$, that
is, if $(Sat(\A),\bigvee^\A,\circ^\A)$ is a commutative quantale.
\end{definition}

Our next aim is to obtain some more elementary characterizations of
this notion.

\begin{lemma}\label{lemma existential op}
For every basic cover with operation $(S,\cov,\circ)$, the following
are equivalent:
\begin{enumerate}
\item $\circ^\A$ distributes over $\bigvee^\A$, that is: 
\enspace $\bigvee^\A_{i\in
I}(\A U_i\circ^\A\A V)$ = $(\bigvee^\A_{i\in I}\A U_i)\circ^\A\A V$;
\item $\circ$
distributes over $\bigcup$ modulo $=_\A$, that is: 
\enspace $\bigcup_{i\in
I}(U_i\circ V)$ $=_\A$ $(\bigcup_{i\in I} U_i)\circ V$;
\item $\circ$ is determined by its restriction on
singletons, that is: 
\enspace $U\circ V$ $=_\A$ $\bigcup_{a \releps U,\ b
\releps V} (\{a\}\circ \{b\})$.
\end{enumerate}
\end{lemma}
\begin{proof}
By unfolding the definitions of $\circ^\A$ in ~(\ref{eq. def. circ^A}) 
and $\bigvee^\A$ in~(\ref{def join SatA}) one sees
that $2$ is just a rewriting of 1. 
Since $U =\bigcup_{a\releps U}\{a\}$ and $V =\bigcup_{b\releps V}\{b\}$,
$3$ follows by
applying 2 twice. It remains to be checked that 3
implies 2: $\bigcup_{i\in I}(U_i\circ V)$ $=_\A$ (by $3$) $\bigcup_{i\in
I}(\bigcup_{a\releps U_i,\ b\releps V}\{a\}\circ \{b\})$  =
$\bigcup_{i\in I,\ a\releps U_i}\bigcup_{b\releps V}\{a\}\circ \{b\}$ =
$\bigcup_{a\releps\bigcup_{i\in I}U_i,\ b\releps V}\{a\}\circ \{b\}$ $=_\A$
(by 3 again) $(\bigcup_{i\in I} U_i)\circ V$.
\end{proof}

>From now on, we write $a\circ b$  for $\{a\}\circ\{b\}$ and more generally   
$a\circ V$ and $U\circ b$ for
$\{a\}\circ V$ and $U\circ \{b\}$, respectively.

When $\circ$ is determined by its restriction on singletons (item  3 in the lemma),  
stability and localization become
equivalent to their particular cases
\begin{equation}
\label{eq. stability and localization on elements}
\frac{a\cov U}{a\circ V\cov U\circ V}\qquad\quad
\frac{a\cov U}{a\circ b\cov U\circ b}\qquad \quad
\frac{a\cov U \quad b\cov V}{a\circ b\cov U\circ V}
\end{equation} 
The equivalence between localization and the second
rule in (\ref{eq. stability and localization on elements}) will be
crucial for inductive generation. 
In a similar way, associativity and
commutativity become equivalent to $(a\circ b)\circ c\cov
a\circ(b\circ c)$ and to $a\circ b\cov b\circ a$, respectively. So
we have:

\begin{proposition}
Let $(S,\cov)$ be basic cover and let $\circ$ be an operation on
$\P(S)$. Then $(S,\cov,\circ)$ is a convergent cover if and only if
all the following hold:
\begin{enumerate}
\item $U\circ V$ $=_\A$ $\bigcup_{a \releps
U,\ b \releps V} (a\circ b)$
\item $a\cov U$ $\land$ $b\cov V$ $\Longrightarrow$ $a\circ b$
$\cov$ $U\circ V$ (or, equivalently: $a\cov U$ $\Longrightarrow$
$a\circ b$ $\cov$ $U\circ b$)
\item $(a\circ b)\circ c$ $\cov$ $a\circ(b\circ c)$
\item $a\circ b$ $\cov$ $b\circ a$
\end{enumerate} (for all $a,b,c\in S$ and $U,V\sub S$).
\end{proposition}

We can characterize convergent covers also as basic covers with
operation plus an extra operation defining implication, as follows.
It is well known that a monotone function on a suplattice, which is
an endofunctor on the corresponding poset category, preserves
arbitrary joins if and only if it admits a right adjoint. This means
that, given a basic cover with operation $(S,\cov,\circ)$, the
structure $(Sat(\A),\bigvee^A,\circ^\A)$ is a quantale, that is
$\circ^\A$ distributes over $\bigvee^\A$, if and only if every map
$\_\_\_\ \circ^\A\A U$ has a right adjoint $\A U\to_\A\_\_\_\ $. In
other words, a basic cover with operation $(S,\cov,\circ)$ is a
convergent cover iff there exists a binary operation $\to_\A$ on
$Sat(\A)$ such that:
$$\A W \circ^\A \A U \sub \A V \quad \Longleftrightarrow \quad \A W
 \sub \A U \to_\A \A V$$ for all $U,V,W\sub S$.
When it exists, $\to_\A$ satisfies:
\begin{equation}
\A U\to_\A\A V\quad=\quad{\bigvee}^\A\big\{\A W\ |\ W\sub S\
\textrm{and}\ \A W\circ^\A\A U\sub\A V\big\}.
\end{equation} 
By unfolding definitions, the right member becomes
$\A(\bigcup\{W\sub S$ $|$ $W\circ U\cov V\})$.

Now we give a predicative definition of $\to_A$. So assume $\to_\A$
to exist; thus $\circ^\A$ distributes over $\bigvee^\A$. 
By lemma~\ref{lemma existential op}, $\circ$ is determined 
by its restriction on singletons.
So $W\circ U\cov V$ becomes equivalent to $(\forall
a\releps W)(a\circ U\cov V)$; hence $\A\bigcup\{W\sub S$ $|$ $W\circ
U\cov V\}$ can be rewritten as $\A\{a\in S$ $|$ $a\circ U\cov V\}$.
Summing up, when $\to_\A$ exists, it is
\begin{equation}
\A U\to_\A\A V\quad=\quad\A\{a\in S\ |\ a\circ U\cov V\}
\end{equation} 
(for all $U,V\sub S$).
This suggests to define an operation
 \begin{equation}\label{def. implication}
 U\to_\S V\quad\stackrel{def}{=}\quad \{a\in S\ |\ a\circ U\cov V\}
 \end{equation}
on \emph{arbitrary} subsets, for every basic cover with operation $\S$ =
$(S,\cov,\circ)$, 
and then to put
 \begin{equation}\label{def. implication SatA}
 \A U\to_\A\A V\quad\stackrel{def}{=}\quad \A(U\to_\S V),
 \end{equation}
which is an analogue of equation (\ref{eq. def. circ^A}).
Then one can show by localization that $\to_\A$ is  well-defined  on $Sat(\A)$.
\begin{proposition}
For every basic cover with operation $\S$ = $(S,\cov,\circ)$, 
 the following are equivalent:

\begin{tabular}{rlp{3pt}rl}
\\
a. & $\S$ is a convergent cover  & & b. & $\to_\A$ is right adjoint
of $\circ^\A$ 
\\
\\ c. & \multicolumn{4}{l}{
$W\circ U\cov V$
$\Longleftrightarrow$ $W\cov U\to_\S V$ (for all $U,V,W\sub S$) \enspace
 with $\to_\S$ defined in (12). }\\
\end{tabular}
\end{proposition}
\begin{proof}
By the discussion above.
\end{proof}

To represent unital commutative quantales \cite{rosenthal}, we also need the
following:
\begin{definition} 
We say that a convergent cover $(S,\cov,\circ)$ is \emph{unital} if
there exists a subset $I\sub S$ such that $a\circ I$ $=_\A$ $a$
$=_\A$ $I\circ a$ for all $a\in S$.
\end{definition}

\subsection{Morphisms between convergent covers}

Let $\S$ = $(S,\cov_\S,\circ_\S)$ and $\T$ = $(T,\cov_\T,\circ_\T)$
be two convergent covers. A morphism between the corresponding
quantales $h$ : $Sat(\A_{\T})$
$\longrightarrow$ $Sat(\A_{\S})$ is a
map which preserves joins and multiplication.
As in proposition~\ref{prop. BCov dual SupLat}, $h$ corresponds to a basic cover map
$r$ from $(S,\cov_\S)$ to $(T,\cov_\T)$. Then the further condition
on $h$ says that $r^-(U\circ_\T V)$ $=_{\A_{\S}}$ $r^-U\circ_\S r^- V$.
This equation is equivalent to its version on singletons since $r^-$ is determined by 
its restriction on singletons. So we put:

\begin{definition}
Let $\S =(S,\cov_\S, \circ_\S)$ and $\T =(T,\cov_{\T},\circ_{\T})$
be two convergent covers. A relation $r$ between $S$ and $T$ is a
\emph{convergent cover map} if:
\begin{list}{-}{ }
\item $r$ is a \emph{basic cover map} (that is, a morphism between basic covers);
\item $r$ is \emph{convergent}, that is
\begin{equation}\label{def. rel. cont. convergente}
\quad r^-(b_1\circ_{\T} b_2)\quad =_\A\quad (r^-b_1)\circ_\S
(r^-b_2)\end{equation} for all $b_1,b_2\in T$.
\end{list} Two convergent cover maps are equal if they are equal as basic cover maps.
\end{definition}
We then specialize the notion of convergent cover map in the
presence of units:
\begin{definition}
Let $\S =(S,\cov_\S, \circ_\S, I_{\S})$ and $\T =(T,\cov_{\T},\circ_{\T}, I_{\T})$
 be two
unital convergent covers. A relation $r$ between $S$ and $T$ is a
\emph{unital convergent cover map} if:
\begin{list}{-}{ }
\item $r$ is a \emph{convergent cover map};
\item $r$ preserves the $\circ$-units, that is
\begin{equation}\label{def. rel. unital}
\quad r^-(I_{\T})\quad =_\A\quad I_{\S}\ .\end{equation}
\end{list} Two unital convergent cover maps are equal if they are equal 
as basic cover maps.
\end{definition}

Given the discussion above, it is straightforward to show that:

\begin{proposition}\label{prop. CCovdual cQu}
Convergent covers with convergent cover maps form a subcategory of
{\bf BCov}, called {\bf CBCov}. The category {\bf CBCov} is dual to
the category {\bf cQu} of commutative quantales.
\end{proposition}

\begin{proposition}\label{prop. CCovdual cQu unital}
Unital convergent covers with unital convergent cover maps form a
subcategory of {\bf CBCov}, called {\bf uCBCov}. The category {\bf
uCBCov} is dual to the category {\bf ucQu} of unital commutative
quantales and commutative quantale maps preserving  units. 
\end{proposition}

The treatment of quantales in \cite{bs} is based on a binary
operation $\bullet$ on the base $S$ of a basic cover. This operation
can be seen as a very particular operation $\circ$ on subsets for
which all $a\circ b$ are singletons. See section \ref{section other
definitions} below for further details.

A quite intuitive fact we shall need later is:

\begin{lemma}\label{lemma circ saturo}
Let $\S_1$ = $(S,\cov,\circ_1)$ and $\S_2$ = $(S,\cov,\circ_2)$ be
two convergent covers (sharing the same underlying basic cover). If
$\circ_1=_\A\circ_2$, that is $U\circ_1 V$ $=_\A$ $U\circ_2 V$ for
all $U,V\sub S$, then $\S_1$ and $\S_2$ are isomorphic (in {\bf
CBCov}).
\end{lemma}
\begin{proof}
The isomorphism is given by the identity relation on $S$.
\end{proof}

\subsection{Inductive generation of convergent covers} %%%%%%%
\label{indqu}

We now wish to extend the method of inductive generation from the
case of basic covers to that of convergent covers. In order to
generate a convergent cover (or, equivalently, a commutative
quantale), it is quite natural to start from the following data: a
set $S$ (that is, a set of generators of the corresponding
suplattice); an axiom-set $I,C$ (encoding axioms of the form $a\cov
U$); a partial description of an operation on subsets given by its
restriction to elements, namely a map $\delta:S\times
S\longrightarrow\P(S)$ (we use a new symbol to underline the fact
that, in the generation of a convergent cover, it is sufficient to
define $\circ$ on singletons).

The first step is to extend $\delta$ to an operation $\circ$ on
$\P(S)$. This is simple: we put $$U\circ
V\quad=\quad\bigcup_{a\releps U,\ b\releps V}\delta(a,b)\ .$$
Recalling that $a\circ b$ stands for the subset $\{a\}\circ\{b\}$,
one sees that $a\circ b$ = $\delta(a,b)$ and hence $U\circ V$ =
$\bigcup_{a\releps U,\ b\releps V}a\circ b$, so that $\circ$ is
determined by its restriction on singletons.

Next we add some conditions making the operation $\circ$ commutative and
associative modulo $=_\A$. 
In order to apply the general scheme of inductive generation of basic covers, 
we are going to express such conditions as instances of the infinity rule.
By using transitivity, one can see that
commutativity is expressed by any one of the following equivalent
conditions:
$$
b\circ c\cov c\circ b\qquad\quad
\frac{c\circ b\cov U}{b\circ c\cov U}\qquad\quad \frac{a\releps
b\circ c \qquad c\circ b\cov U}{a\cov U}\quad .
$$
We choose the last one
since it  becomes an instance of the infinity rule provided that the axiom schema
``$a\cov c\circ b$ whenever $a\releps b\circ c$'' is encoded in the
axiom-set. To this aim, it is sufficient first to enlarge the index
set $I(a)$ by adding all pairs $(b,c)$ such that $a\releps b\circ c$
and then to define the corresponding cover of $a$ to be $c\circ b$.

Associativity is treated by following the same idea. One can see
that
$$(b\circ c)\circ d\cov b\circ(c\circ d)\qquad\quad
\frac{b\circ(c\circ d)\cov U}{(b\circ c)\circ d\cov U}\qquad\quad
\frac{a\releps (b\circ c)\circ d \qquad b\circ(c\circ d)\cov
U}{a\cov U}$$ are all equivalent. The last one becomes an instance
of the infinity rule for a suitable extension of the axiom-set
$I,C$.

We call $J,D$ the axiom-set extending $I,C$ in the way just
described. By the above equivalences, the basic cover generated by
$J,D$ is the least basic cover which makes $\circ$ commutative and 
associative modulo $=_\A$.

Thus it only remains to take care of localization.
It is convenient to express it in the equivalent form given by the
second rule in (\ref{eq. stability and localization on elements}).
In fact, as we now see, this allows us to show that localization can
equivalently be expressed by a set-indexed family of conditions. A
straightforward argument shows that the rules
$$\infer{b\circ c\cov V\circ c}{b\cov V}\qquad
\infer{b\circ c\cov U}{b\cov V & V\circ c\cov U}\qquad\infer{a\cov
U}{a\releps b\circ c & b\cov V & V\circ c\cov U}$$ are all
equivalent. The last one looks more suitable for an inductive
generation, since it resembles the infinity rule. However, it is not
acceptable from a constructive point of view since the parameter $V$
ranges over a collection, namely $\P(S)$. Then the idea is to
restrict the cover $V$ of $b$ to be one of those given by the
axioms, namely $D(b,j)$ for $j\in J(b)$. This leads to the following
rule:
$$
\infer[\enspace locax]{a\cov U}{a\releps b\circ c & j\in J(b) & D(b,j)\circ c\cov
U}
$$ 
(\emph{localization on axioms}). This rule becomes an instance
of the infinity rule for a suitable extension $J',D'$ of the
axiom-set $J,D$. In fact, it is sufficient to enlarge the index set
$J(a)$ by adding triples $(b,j,c)$ such that $a\releps b\circ c$ and
$j\in J(b)$ and then to define the corresponding cover of $a$ to be
$D(b,j)\circ c$. The rule locax is equivalently
expressed by 
$$
\infer{b\circ c\cov U}{j\in J(b) & D(b,j)\circ c\cov U}
$$ 
which explains its name: every instance of the infinity rule
for $J,D$ is ``localized'' to the basic neighbourhood $c$. Using the
fact that $\circ$ is determined by its restriction on singletons,
one can easily check that locax holds also for subsets, that is
$$
\infer[\enspace locax \ on \ subsets]{b\circ V\cov U}{j\in J(b) & D(b,j)\circ V\cov U}
$$ 
for every $V\sub S$.

It is worth noting that the rule locax cannot be limited to the
axioms of $I,C$. If that were the case, in fact, it would become
impossible to prove the localized versions of commutativity and
associativity. For instance, to prove $(a\circ b)\circ c$ $\cov$
$(b\circ a)\circ c$ it is necessary to apply locax with respect to
$J,D$: for $x\releps a\circ b$ one has $(a,b)\in J(x)$ and $D(x,(a,b)) = b \circ a$
and hence from $(b\circ a)\circ c \cov (b\circ a)\circ c$ by locax one gets
$x \circ c \cov (b\circ a)\circ c$.
An alternative approach would be to consider
locax for $I,C$ and then add commutativity and associativity already
in localized form.

Summing up, starting from any axiom-set 
$I,C$ on a set $S$ and any map 
$\delta$ $:$ $S\times S$ $\longrightarrow$ $\P(S)$,  we first extended  $\delta$ to a map 
$\circ:$ $\P(S)\times \P(S)$ $\longrightarrow$ $\P(S)$
by putting
$U\circ V$ = $\bigcup\{\delta(a,b)$ $|$ $a\releps U$, $b\releps V\}$ for all $U,V\sub S$.
Then we called  $J,D$ the  axiom-set  obtained from $I,C$ by adding axioms
so that
$$
b\circ c  \cov c\circ b~\qquad \mbox{ and } \qquad (b\circ c)\circ d\cov b\circ(c\circ d)
$$
become derivable.  
Finally, we let
$J',D'$ be the axiom-set extending   $J,D$ with axioms making
$$
b\circ c  \cov D(b,j)\circ c
$$ 
derivable for all $b,c\in S$ and $j\in J(b)$.
Then we can prove:

\begin{proposition}\label{prop. ind. gen. of convergent
covers} Let $I,C$ be an axiom-set on a set $S$ and let $\delta$ $:$
$S\times S$ $\longrightarrow$ $\P(S)$ be an arbitrary map. 
Define $\circ$ and $J',D'$ as above and  let $\cov$ be 
the basic cover generated by  $J',D'$. 
Then $(S,\cov,\circ)$ is a convergent cover
in which $\cov$ contains $\cov_{I,C}$ (that is,
$a\cov_{I,C}U$ $\Rightarrow$ $a\cov U$ for all $a\in S$ and $U\sub S$)
and $\circ$ extends $\delta$ (that is, $a\circ b$ $=_{\A}$
$\delta(a,b)$ for all $a,b\in S$).

Moreover,  if $(S,\cov',\circ')$  is any convergent cover in which $\cov'$ contains 
$\cov_{I,C}$  and $\circ'$ extends $\delta$, then $\cov'$ contains $\cov$ and 
$\circ' =_{\A'} \circ$.
 \end{proposition}
\begin{proof}
The operation $\circ$ trivially extends $\delta$ and satisfies
$U\circ V$ $=_\A$ $\bigcup_{a \releps
U,\ b \releps V} (a\circ b)$.
By the definition of $J',D'$, the basic cover $\cov$ contains $\cov_{I,C}$ and 
associativity and commutativity hold. 
So, to show that  $(S,\cov,\circ)$ is a convergent cover, only localization 
remains to be proved. We
prove $a\cov U$ $\Longrightarrow$ $a\circ d\cov U\circ d$ by
induction on the proof of $a\cov U$.

If $a\cov U$ is obtained by reflexivity from $a\releps U$, then by
definition of $\circ$ we have $a\circ d$ $\sub$ $U\circ d$ and hence
$a\circ d\cov U\circ d$ by reflexivity.

If $a\cov U$ is obtained by infinity, we consider two cases
according to whether the axiom used in the rule belongs to $J,D$ or
not. In the former case, $a\cov U$ is obtained from the assumptions:
$j\in J(a)$ and $D(a,j)\cov U$. By the inductive hypothesis applied
to $D(a,j)\cov U$, we get $D(a,j)\circ d\cov U\circ d$
(pedantically, for each $b\releps D(a,j)$, we use the inductive
hypothesis $b\cov U$ $\Rightarrow$ $b\circ d\cov U\circ d$). This,
together with $j\in J(a)$, implies $a\circ d\cov U\circ d$ by locax
on subsets (that is, by a suitable instance of the infinity rule for
$J',D'$).

We now analyze the case in which $a\cov U$ is obtained by infinity
from an axiom of $J',D'$ that does not belong to $J,D$. In other
words, $a\cov U$ is derived from the assumptions $a\releps b\circ
c$, $j\in J(b)$ and $D(b,j)\circ c\cov U$ and hence the infinity
rule corresponding to this case is precisely locax. The inductive
hypothesis on the assumption $D(b,j)\circ c\cov U$ gives
$(D(b,j)\circ c)\circ d\cov U\circ d$ and hence $D(b,j)\circ (c\circ
d)\cov U\circ d$ by associativity. This, together with $j\in J(b)$,
yields $b\circ (c\circ d)\cov U\circ d$ by locax on subsets and
hence $(b\circ c)\circ d\cov U\circ d$ by associativity again. So
$a\circ d\cov U\circ d$ as wished, because $a\releps b\circ c$ and
$\circ$ is determined by its restriction on singletons.

Finally, let $(S,\cov',\circ')$  be a convergent cover in which $\cov'$ contains 
$\cov_{I,C}$  and such that $a \circ' b =_{\A'}  \delta(a,b)$.
Then $U\circ' V =_{\A'} \bigcup_{a\releps U, b\releps V} \delta(a,b) = 
U\circ V$. So also $(S,\cov',\circ)$ is a convergent cover 
(in fact, it is isomorphic to $(S,\cov',\circ')$ by lemma~\ref{lemma circ saturo})
and hence $\cov'$ fulfills all the axioms in $J',D'$. So $\cov'$ contains $\cov$.
\end{proof}

\begin{remark} In order to generate a unital convergent cover,
it is sufficient to start with an additional piece of data, namely a
subset $I$, and then impose  extra axioms about unit $a\cov a\circ I$ and $ a\circ I\cov a$.
Note that, in the presence of commutativity and associativity, it is
irrelevant whether the unit axioms added before or after
localizing the axioms.
\end{remark}

\begin{definition}
A convergent cover $(S,\cov,\circ)$ is \emph{inductively generated}
if it is constructed as in proposition \ref{prop. ind. gen. of
convergent covers} for some axiom-set $I,C$ over $S$ and some map
$\delta$ $:$ $S\times S$ $\longrightarrow$ $\ P(S)$.

We call {\bf CBCov$_i$} the full subcategory of {\bf CBCov} whose
objects are inductively generated.
Similarly,  {\bf uCBCov$_i$} is the full subcategory of inductively 
generated objects in {\bf uCBCov}.
\end{definition}

We end with a lemma characterizing convergent cover maps between
inductively generated unital convergent covers:
\begin{lemma}\label{lemma continuity for generated convergent covers}
Let $\S$ = $(S,\cov_{\S},\circ_{\S}, I_{\S})$ and $\T$ =
$(T,\cov_{\T}, \circ_{\T}, I_{\T})$ be two unital convergent covers. 
Assume that $\T$ is inductively generated by means of an
axiom-set $I,C$ and a map $\delta:T\times T\longrightarrow\P(T)$
according to proposition \ref{prop. ind. gen. of convergent covers}.
Then   a relation $r$ between $S$ and $T$  is a unital
convergent cover map from $\S$ to $\T$ if and only if the following
hold:
\begin{enumerate}
\item $r^-a\cov_{\S} r^-C(a,i)$ for all $a\in T$ and all $i\in I(a)$;
\item $r^-\delta(a,b)\ =_\A\ r^-a\circ_{\S}r^-b$ for all $a,b\in T$;
\item $r^-I_{\T}\ =_\A\ I_{\S}$.
\end{enumerate}
\end{lemma}
\begin{proof} 
By the definition of $\circ_{\T}$ 
(in proposition \ref{prop. ind. gen. of convergent covers}), item $2$ is equivalent to
$r^-(a\circ_{\T} b)$ $=_\A$ $(r^-a)\circ_{\S}(r^-b)$. So we must
check only that condition $1$ is tantamount to $r$ being a basic
cover map. 
By lemma~\ref{lemma continuity for generated covers},
it is sufficient to show that $1$ holds   
also for the other axioms of  $J',D'$,
namely commutativity, associativity and localization.
In fact, condition $1$ for these extra axioms follows from $2$ and the
corresponding properties of $\S$. 
We check this only in two cases.

First, we see that
$r^-a\cov_\S r^-(c\circ_\T b)$ whenever $a\releps b\circ_\T c$,
namely that condition $1$ for one of the commutativity axioms
of $J,D$. By $2$ and commutativity of $\circ_\S$  we have
$r^-(c\circ_\T b)$ $=_\A$ $(r^-c)\circ_\S(r^-b)$ $=_\A$
$(r^-b)\circ_\S(r^-c)$ $=_\A$ $r^-(b\circ_\T c)$. So the claim is
equivalent to $r^-a\cov r^-(b\circ_\T c)$. But this follows from
$r^-a\sub r^-(b\circ_\T c)$ which in  turn is a consequence of
the assumption $a\releps b\circ_\T c$.

Second, we
consider a particular instance of locax. Assume $a\releps
b\circ_{\T} c$ and $C(b,i)\circ_{\T} c\cov_{\T} U$ for some $b,c\in
T$ and some $i\in I(b)$. Then $r^-(C(b,i)\circ_{\T} c)\cov_{\S}
r^-U$ by inductive hypothesis and hence $(r^-C(b,i))\circ_\S
(r^-c)\cov_{\S} r^-U$ by $2$. From $1$ by localization in
$\S$ one can deduce $(r^-b)\circ_{\S} (r^-c)\cov_{\S}
(r^-C(b,i))\circ_{\S}  (r^-c)$. So by transitivity $(r^-b)\circ_{\S} (r^-c)\cov_{\S}
r^-U$ and hence $r^-(b\circ_{\T} c)\cov_{\S}  r^- U$ by 2. Since
$a\releps b\circ_{\T} c$ by assumption, one can conclude
$r^-a\cov_{\S}  r^- U$.
\end{proof}

\subsection{Categorical reading of convergence}

In this section we wish to show that the category {\bf CBCov$_i$} is
equivalent to the category of commutative co-semigroups in {\bf
BCov$_i$}. This result expresses in a constructive way the fact that
{\bf cQu} (commutative quantales) is equivalent to the category of
commutative semigroups in {\bf SupLat} (see \cite{galoisexten}).
Once this is achieved, it is straightforward to obtain a description of
commutative co-monoids in {\bf BCov$_i$} simply by 
considering unital convergent covers. Similarly,
one can easily extends these results to the non-commutative case.

\begin{definition}
A commutative co-semigroup in the category  {\bf BCov$_i$} is an
inductively generated basic cover $\S$ together with a map $\mu:\S
\rightarrow \S \otimes \S$ in {\bf BCov$_i$} such that the following
diagrams commute.
$$\xymatrix{ & \S \ar[dl]_{\mu} \ar[dr]^{\mu} & \\
\S\otimes\S \ar[d]_{id_{\S}\otimes\mu} & & \S\otimes\S \ar[d]^{\mu\otimes id_{\S}} \\
\S\otimes(\S\otimes\S) \ar[rr]_{\alpha_{\S,\S,\S}} & &
(\S\otimes\S)\otimes\S}
\qquad\xymatrix{ & \S \ar[dl]_{\mu} \ar[dr]^{\mu} & \\
\S\otimes\S \ar[rr]_{\gamma_{\S,\S}} & & \S\otimes\S}$$ A morphism
between two commutative co-semigroups $(\S,\mu)$ and $(\S',\mu')$ is
a basic cover map $r:\S\longrightarrow\S'$ such that the following
diagram commutes.
$$\xymatrix{ \S \ar[r]^{r} \ar[d]_{\mu} & \S' \ar[d]^{\mu'}  \\
\S\otimes\S \ar[r]_{r\otimes r} & \S'\otimes\S'}$$
\end{definition}

\begin{lemma}
Every  inductively generated convergent cover
 can be seen as a commutative co-semigroup in {\bf BCov$_i$}.

Conversely every commutative co-semigroup in {\bf BCov$_i$} determines a
convergent basic cover.
\end{lemma}
\begin{proof}
Given a convergent cover $\S\, =\, (S, \cov, \circ)$, let us define
$$
\mu_\circ:\S\longrightarrow  \S\otimes \S\qquad \mbox{ as }\qquad c\,\mu_\circ (a,b)
\,\equiv \, c\releps a\circ b
$$ 
for all $c,a,b\in S$.
That is, we put ${\mu_\circ}^- (a,b)$ =
$a\circ b$. In order to prove that $\mu_\circ$ is a basic cover map,
by lemma~\ref{lemma continuity for generated covers} 
it is sufficient to check that ${\mu_\circ}^-(a,b)\cov {\mu_\circ}^-
D((a,b),j)$ holds for every axiom $(a,b)\cov D((a,b),j)$ of $\S\otimes
\S$. By symmetry, we can consider only the case $D((a,b),j)$ =
$C(a,i)\times\{b\}$ for some axiom $a\cov C(a,i)$ of $\S$. 
>From $a\cov C(a,i)$  by localization one has $a\circ b\cov C(a,i)\circ b$ and
so:
$$
{\mu_\circ}^- (a,b)= a\circ b\,\cov\, C(a,i)\circ b=_\A \bigcup_{c\releps C(a,i)}c\circ
b\ .
$$ 
Since $c\circ b$ = ${\mu_\circ}^-(c,b)$ and since ${\mu_\circ}^-$
distributes over unions, one gets 
$$
\bigcup_{c\releps
C(a,i)}c\circ b= {\mu_\circ}^-(\bigcup_{c\releps C(a,i)}\{(c,b)\})=
{\mu_\circ}^-(\, C(a,i)\times \{b\}\, ).
$$
So we can conclude that ${\mu_\circ}^-(a,b)\cov {\mu_\circ}^-
D((a,b),j)$.

Commutativity and associativity of $((S,\cov),\mu_\circ)$ follow from
the fact that $\circ$ is commutative and associative modulo $=_\A$.
For instance, the equality
$\gamma_{\S,\S}\cdot\mu_{\circ}=_{\mathbf{BCov}_i}\mu_{\circ}$ means
 that ${\mu_\circ}^- {\gamma_{\S,\S}}^-(a,b)$ $=_\A$
${\mu_\circ}^-(a,b)$ for all $(a,b)\in S\times S$; by unfolding
definitions and since $\mu_\circ$ respects covers, this
becomes ${\mu_\circ}^-(b,a)$ $=_\A$ $a\circ b$, that is, $b\circ a$
$=_\A$ $a\circ b$.

Conversely, given a commutative co-semigroup with $\mu:\
\S\longrightarrow  \S\otimes \S$, we put
$$U\circ_\mu V\, =\, \A\mu^-(U\times V)$$ for all $U,V\sub S$. Note
that this definition respects the equality of morphisms (definition
\ref{def.basic cover map}).

By (\ref{eq saturazione prodotto}), 
$\A U\times\A V$ $=_{\A_{\S\otimes\S}}$ $U\times V$ holds
and hence $\mu^-(\A U\times\A V)$
$=_\A$ $\mu^-(U\times V)$ because $\mu$ respects covers. By the
definition of $\circ_\mu$, this means that $\A U\circ_\mu \A V$
$=_\A$ $U\circ_\mu V$, so that $\circ_\mu$ is well-defined
(proposition \ref{primi equivalenti di stability}). The operation
$\circ_\mu$ is determined by its restriction on singletons:
$$
U\circ_\mu V=_\A\mu^-(\, U\times V\, )=
\bigcup_{u\releps U,v\releps V}\mu^-(u,v)=_\A \bigcup_{u\releps
U,v\releps V}u\circ_\mu v\ .
$$ 
Finally, $\circ_\mu$ is associative
and commutative modulo $=_\A$ because  so is the co-semigroup. For
instance, by the equation
$\alpha_{\S,\S,\S}\cdot (\mathsf{id}_\S\otimes\mu)\cdot \mu$ =
$(\mu\otimes\mathsf{id}_\S)\cdot \mu$ one has
$$
\mu^-(\mathsf{id}_\S\otimes\mu)^-\alpha_{\S,\S,\S}^-((U\times V)\times W)
\quad =_\A\quad \mu^-(\mu\otimes\mathsf{id}_\S)^-((U\times V)\times W)
$$ 
for all $U,V,W\sub S$. By the definition of
$\alpha_{\S,\S,\S}$ and of the tensor of two morphisms, this gives
$\mu^-(U\times \mu^-(V\times W))$ $=_\A$ $\mu^-(\mu^-(U\times V)\times
W))$. By the definition of $\circ_\mu$ this is precisely
$U\circ_\mu (V\circ_\mu W)$ $=_\A$ $(U\circ_\mu V)\circ_\mu W$.
\end{proof}

\begin{proposition}
The category {\bf CBCov$_i$} of inductively generated convergent
covers is equivalent to the category of commutative co-semigroups
in {\bf BCov$_i$}.
\end{proposition}
\begin{proof} Thanks to the previous lemma, to each object
$(S,\cov,\circ)$ in {\bf CBCov$_i$} we associate the commutative
co-semigroup $((S,\cov),\mu_\circ)$. Conversely, to every
commutative co-semigroup $((S,\cov),\mu)$ we associate
$(S,\cov,\circ_\mu)$. By lemma \ref{lemma circ saturo},
$(S,\cov,\circ)$ is isomorphic to $(S,\cov,\circ_{\mu_\circ})$
because $U\circ_{\mu_\circ}V$ = $\A{\mu_\circ}^-(U\times V)$ =
$\A\bigcup_{a\releps U,\ b\releps V}{\mu_\circ}^-(a,b)$ =
$\A\bigcup_{a\releps U,\ b\releps V}a\circ b$ = $\A(U\circ V)$.
Moreover, $((S,\cov),\mu)$ coincides with
$((S,\cov),\mu_{\circ_\mu})$ because ${\mu_{\circ_\mu}}^-(a,b)$ =
$a\circ_\mu b$ = $\A\mu^-(a,b)$, for all $a,b\in S$ and so
$\mu_{\circ_\mu}=_{\mathbf{BCov}_i}\mu$.

Morphisms between co-semigroups correspond to
convergent cover maps. In fact, a basic cover map $r:
\S\longrightarrow \S'$ is a co-semigroup map from $(\S,\mu)$ to
$(\S',\mu')$ if and only if $\mu'\cdot r=_{\mathbf{BCov}_i}(r\otimes r)\cdot \mu$ 
iff $r^-{\mu_{\S'}}^-(a,b)$ $=_\A$ ${\mu_{\S}}^-(r\otimes
r)^-(a,b)$ for all $a,b\in S'$. This means precisely that
$r^-(a\circ_{\mu'} b)$ $=_\A$ $r^-a\circ_\mu r^-b$, that is $r$ is a
convergent cover map from $(S,\cov,\circ_\mu)$ to
$(S',\cov',\circ_{\mu'})$. Vice versa, $r$ : $(S,\cov,\circ)$
$\longrightarrow$ $(S',\cov',\circ')$ is a convergent cover map iff
$r^-(a\circ' b)$ $=_\A$ $r^-a\circ r^-b$, that is,
$r^-{\mu_{\circ'}}^-(a,b)$ $=_\A$ ${\mu_{\circ}}^-(r\otimes r)^-(a,b)$.
This means that $\mu_{\circ'}\cdot r=_{\mathbf{BCov}_i}(r\otimes
r)\cdot\mu_\circ$, that is $r$ is a morphism of co-semigroups from
$((S,\cov),\mu_\circ)$ to $((S',\cov'),\mu_{\circ'})$.
\end{proof}

The restriction to inductive generated structures in the
previous statement is needed only to be sure that $\otimes$ exists
predicatively. In an impredicative framework, the above result
states that the whole category {\bf CBCov} is equivalent to the category of
commutative co-semigroups in {\bf BCov}. By passing to the
opposite categories, one gets that (commutative) quantales are
equivalent to (commutative) semigroups over suplattices.

\begin{definition}
A commutative co-monoid in the category  {\bf BCov$_i$} is a
commutative co-semigroup $(\S,\mu)$ in {\bf BCov$_i$} together with
a map $\eta:\S \rightarrow E$  such that the following diagram
commutes.
$$\xymatrix{ & \S \ar[dl]_{{\lambda_\S}^{-1}} \ar[d]^{\mu} \ar[dr]^{{\rho_{\S}}^{-1}} & \\
E\otimes\S & \S\otimes\S \ar[l]^{\eta\otimes id_\S}
\ar[r]_{id_\S\otimes\eta} & \S\otimes E }$$
\end{definition}

Analogously to the previous result, one can show that:

\begin{corollary}
The category $\mathbf{uCBCov}_i$ of inductively generated unital
convergent covers is equivalent to the category of commutative
co-monoids in $\mathbf{BCov}_i$.
\end{corollary}

\section{Predicative locales: formal covers}
\label{fcov} In this section we finally come to the case of locales. 
It is common to associate the name ``Formal Topology''
with a predicative version of the theory of Locales. However,
following its first appearance in \cite{IFS}, we prefer to conceive
formal topologies as a predicative version of \emph{open} (also
known as \emph{overt}) locales (see \cite{galoisexten}). Predicative
presentations of locales tout court are called here formal covers. 

As usual a \emph{locale} or \emph{frame} is a suplattice in which binary 
meets distribute over arbitrary joins.
It is well known that a quantale is a locale exactly when its
multiplication coincides with the order-theoretic meet. In our
framework, the quantale $(Sat(\A),\bigvee^\A,\circ^\A)$ presented by
a convergent cover $(S,\cov,\circ)$ is actually a locale when
$$\A U\circ^\A\A V\quad=\quad \A U\wedge^\A\A V$$ for every $U,V\sub
S$. By unfolding definitions, this reduces to $\A (U\circ V)= \A U
\cap \A V $ for every $U,V\sub S$.

\begin{definition}
A convergent cover $(S,\cov,\circ)$ is called a \emph{formal cover}
if $\circ^\A$ = $\wedge^\A$, that is: $$\A(U\circ V)\quad =\quad \A
U\cap\A V$$ for all $U,V\sub S$. In this case,
$(Sat(\A),\bigvee^\A,\circ^\A)$ is a locale.
\end{definition}

\begin{lemma}
\label{equivcircmeet} For every convergent cover all the items in
each column are equivalent one to another:

\begin{tabular}{ll@{$\qquad$}|@{$\qquad$}ll} \\
$A1.$ & $\A(U \circ V) \sub \A U \cap \A V$ & $B1.$ & $\A U \cap \A
V \sub \A (U \circ V)$ \\ & & & \\ $A2.$ & $U \circ V \cov U$ and $U
\circ V \cov V$ \hfill
(weakening) & $B2.$ & $U \cov U \circ U$ \hfill (contraction) \\
& & & \\ $A3.$ & $\displaystyle{\frac{U \cov W}{U \circ V \cov W}}$
and $\displaystyle{\frac{V \cov W}{U \circ V \cov W}}$ \hfill
($\circ$-left) & $B3.$ & $\displaystyle{\frac{W\cov U \quad W\cov
V}{W\cov U \circ V}}$ \hfill ($\circ$-right) \\ & & &
\\ $A4.$ & $a \circ b \cov a$ and $a \circ b \cov b$ \hfill (weakening)
& $B4.$ & $a \cov a \circ a$ \hfill (contraction) \\ & & & \\
$A5.$ & $\displaystyle{\frac{a \cov W}{a \circ b \cov W}}$ and
$\displaystyle{\frac{b \cov W}{a \circ b \cov W}}$ \hfill
($\circ$-left) & $B5.$ & $\displaystyle{\frac{a\cov U \quad a\cov
V}{a\cov U \circ V}}$ \hfill ($\circ$-right)
\end{tabular}

\vskip10pt
\noindent Summing up, the following are equivalent:
\begin{list}{-}{ }
\item[$i)$] $\circ^{\A}$ = $\wedge^{\A}$, that is $\A (U \circ V)$ = $\A U\cap \A V$;
\item[$ii)$] weakening and contraction axioms;
\item[$iii)$] $\circ$-left and $\circ$-right.
\end{list}
\end{lemma}
\begin{proof}
Left to the reader.
\end{proof}

This leads to a number of characterizations of formal covers. For
instance, formal covers are precisely the convergent covers
satisfying weakening and contraction. We have also the following.

\begin{proposition}\label{prop. charact. formal cover}
A structure $\S = (S,\cov,\circ)$ is a formal cover if and only if
$(S,\cov)$ is a basic cover and $\circ$ is a binary operation on
subsets of $S$ such that:
\begin{enumerate}
\item $U\circ V$ $=_\A$ $\bigcup_{a \releps
U,\ b \releps V} (a\circ b)$\hfill ($\circ$ is determined by its
restriction on singletons)
\item $a\cov U$ $\land$ $a\cov V$ $\Longrightarrow$ $a$ $\cov$ $U\circ V$\hfill($\circ$-right)
\item $a\circ b$ $\cov$ $a\ $ and $\ a\circ b$ $\cov$ $b$\hfill(weakening)
\end{enumerate} for all $a,b,c\in S$ and $U,V\sub S$.
\end{proposition}
\begin{proof}
Stability, associativity and commutativity are all
derivable from  $\circ$-right and weakening. In fact,
if $a\cov U$ and $b\cov V$, then $a\circ b\cov U$ and $a\circ b\cov
V$ by weakening (and transitivity); so $a\circ b\cov U\circ V$ by
$\circ$-right. This proves stability. Commutativity is easy: since
$a\circ b$ is covered both by $b$ and $a$ (weakening), it is also
covered by $b\circ a$ ($\circ$-right). Finally, we check
associativity in the form $(a\circ b)\circ c\cov a\circ(b\circ c)$.
By $\circ$-right it is sufficient to prove $(a\circ b)\circ c\cov a$
and $(a\circ b)\circ c\cov b\circ c$. The former follows by
combining $(a\circ b)\circ c\cov a\circ b$ and $a\circ b\cov a$
which both hold  by weakening. To prove the latter it is
sufficient to have both $(a\circ b)\circ c\cov b$ and $(a\circ
b)\circ c\cov c$ which, again, follow by weakening.
\end{proof}

Note that every formal cover is unital:
\begin{proposition}\label{prop. every formal cover is unital}
If $\S
=(S,\cov,\circ)$ is a formal cover, then  $S$ is a unit for $\circ$.
\end{proposition}
\begin{proof}
$a\circ S$ $=_\A$ $a$ holds  thanks to weakening and contraction.
\end{proof}

This proposition says that, impredicatively, every frame is a unital quantale.

\subsection{Morphisms between formal covers}

Morphisms between frames are functions
preserving both arbitrary joins and finite meets.
The category of frames is called {\bf Frm}. The category of locales {\bf Loc} 
is usually introduced as the opposite of {\bf Frm}. So objects are the same, 
while morphisms of locales are usually just morphisms of frames
``in the opposite direction''.
Contrary to this common practice, one can provide
an {\it intuitive topological} definition of morphisms between 
formal covers (see~\cite{MV04,bp}).

When frames are seen as particular unital quantales, their morphisms 
are precisely unital quantale morphisms.
In our framework, thanks to 
proposition~\ref{prop. every formal cover is unital}, we put:
\begin{definition}
The full subcategory of {\bf uCBCov} whose objects are formal covers  is 
called {\bf FCov}.
A unital convergent cover map between
two formal covers is called a \emph{formal cover map}.
\end{definition}

Thus a formal cover map $r\ :\ \S\rightarrow\T$ is \emph{total}, that is, it satisfies the equation
$r^-T$ $=_{\A_\S}$ $S$.  

\begin{proposition}
{\bf FCov} is impredicatively dual to {\bf Frm} and equivalent to
{\bf Loc}.
\end{proposition}
\begin{proof}
Recall that {\bf CBCov} is dual to {\bf cQu} (proposition \ref{prop.
CCovdual cQu}). So it is sufficient to observe that a convergent
cover is a formal cover iff the corresponding quantale is a  locale,
and that a convergent cover map is total iff the corresponding
morphisms of quantales preserves top elements.
\end{proof}

If $\S$ = $(S,\cov,\circ)$ and $\S'$ = $(S,\cov,\circ')$ are two formal covers
with the same underlying basic cover, then $U \circ V$ $=_\A$ $U \circ' V$ for all $U,V \sub S$
by the definition of formal cover.
So $\S$ and $\S'$ are isomorphic as convergent covers by lemma
\ref{lemma circ saturo}. Since the isomorphism provided by that
lemma is given by the identity relation on $S$, it is also total and
hence it is an isomorphism of formal covers. This proves the
following:

\begin{lemma}\label{lemma uniqueness of circ for formal covers}
Two formal covers sharing the same underlying basic cover are
isomorphic. In other words, given a basic cover $(S,\cov)$, there
exists, up to isomorphism of formal covers, at most one operation
$\circ$ such that $(S,\cov,\circ)$ is a formal cover.
\end{lemma}

We will see (lemma \ref{lemma circ eq. fregiu}) that, when it
exists, the operation  $\circ$ of a formal cover coincides up to $=_\A$ with an operation
that can be characterized explicitly in terms of the cover (see
(\ref{eq. def. fregiu indotto}) below). Hence, while for convergent
covers the operation $\circ$ is \emph{new structure}, for a formal
cover the existence of $\circ$ becomes a \emph{property}.

\subsection{Inductive generation of formal covers}
\label{inductive generation formal covers}

In this section we extend the method for generating basic covers
(suplattices) and convergent covers (commutative quantales) to the
case of formal covers (locales). Given an axiom-set $I,C$
and a map $\delta$ as in proposition \ref{prop. ind. gen. of
convergent covers}, we consider the axiom-set $J',D'$ constructed
there. Here we define a further axiom-set, say $J'',D''$, by
enlarging $J',D'$ in a suitable way in order to take care of the
extra axiom schemata $a\circ b\cov a$ (weakening) and $a\cov a\circ
a$ (contraction).

\begin{definition}
\label{axfc}
Let $I,C$ be an axiom-set on a set $S$ and let $\delta$ $:$
$S\times S$ $\longrightarrow$ $\P(S)$ be an arbitrary map. We put
$U\circ V$ = $\bigcup\{\delta(a,b)$ $|$ $a\releps U$, $b\releps V\}$.
We call $J'',D''$ the axiom-set obtained by enlarging the axiom-set $J',D'$ 
in section~\ref{indqu} with the axioms of the form
$$
b\circ c  \cov c\qquad \qquad a\cov a\circ a
$$ for $a, b,c$ in $S$.
Formally, for every $a\in S$, $J''(a)$ is obtained from $J'(a)$ by adding a new index 
for every pair $(b,c)$ such that  $a\releps b\circ c$ and a single further index $*$.
Then one adds respectively the axioms $a\cov c$ for every $(b,c)$ and $a\cov a\circ a$.  
\end{definition}

\begin{proposition}\label{prop. ind. gen. of formal covers}
Let $I,C$ be an axiom-set on a set $S$ and let $\delta$ $:$ $S\times
S$ $\longrightarrow$ $\P(S)$ be an arbitrary map. We put $U\circ V$
$=$ $\bigcup\{\delta(a,b)$ $|$ $a\releps U$, $b\releps V\}$ and we
consider the basic cover $\cov$ generated by the axiom-set $J'',D''$
described above. Then $(S,\cov,\circ)$ is the least formal cover
 containing $\cov_{I,C}$ (that is, $a\cov_{I,C}U$ $\Rightarrow$
$a\cov U$ for all $a\in S$ and $U\sub S$) and extending $\delta$
(that is, $a\circ b$ $=_\A$ $\delta(a,b)$ for all $a,b\in S$).
\end{proposition}
\begin{proof}
The claim is almost obvious after proposition \ref{prop. ind. gen. of convergent covers}. One
should only note that it is not necessary to modify locax in order to
take care of the new axioms of weakening and contraction. For
instance, the localized form of weakening, namely $(a\circ b)\circ
c$ $\cov$ $a\circ c$, follows already from its standard form. In
fact, by associativity and commutativity, it is equivalent to
$(a\circ c)\circ b$ $\cov$ $a\circ c$ which holds by
weakening.
\end{proof}

\begin{definition}
A formal cover $(S,\cov,\circ)$ is \emph{inductively generated} if
it is constructed as in proposition \ref{prop. ind. gen. of formal
covers} for some axiom-set $I,C$ over $S$ and some map $\delta$ $:$
$S\times S$ $\longrightarrow$ $\ P(S)$.

We call {\bf FCov$_i$} the full subcategory of {\bf FCov} whose
objects are inductively generated.
\end{definition}

It is straightforward to extend 
lemma~ \ref{lemma continuity for generated convergent covers} to the framework 
of  formal covers.
\begin{lemma}\label{lemma continuity for generated formal covers}
Let $\S$ = $(S,\cov_{\S},\circ_{\S})$ and $\T$ = $(T,\cov_{\T},\circ_{\T})$ be two
formal covers. Assume that $\T$ is inductively generated by means of
an axiom-set $I,C$ and a map $\delta$ $:$ $T\times T$
$\longrightarrow$ $\P(T)$ according to proposition \ref{prop. ind.
gen. of formal covers}. Finally, let $r$ be a relation between $S$
and $T$. Then $r$ is a formal cover map from $\S$ to $\T$ if and only
if the following hold:
\begin{enumerate}
\item $r^-a\cov r^-C(a,i)$ for all $a\in T$ and all $i\in I(a)$;
\item $r^-\delta(a,b)\ =_\A\ (r^-a)\circ(r^-b)$ for all $a,b\in T$;
\item $r^-T\ =_\A\ S$.
\end{enumerate}
\end{lemma}

The following picture summarizes the main definitions of this paper.
$$
\label{tabella}
\begin{array}{|c|c|c|}
\hline & STRUCTURE & MORPHISMS \\
 & ON\ OPENS & \\ \hline
 \textbf{basic cover} & \mbox{suplattice} & \mbox{basic cover
 map =}\\
 \S=(S,\cov) & (Sat(\A),\bigvee^\A) & r:S\rightarrow T \mbox{ relation } + \\
  &  & b\cov_\T V\Rightarrow r^-b\cov_\S r^- V \\
  \hline 
\textbf{convergent cover} &
\mbox{commutative quantale} & \hfill\mbox{convergent cover map =} \\
 \S=(S,\cov,\circ) & (Sat(\A),\bigvee^\A,\circ^\A) & \mbox{basic cover map } + \\
 \mbox{basic cover + associativity +}&  & r^-(a
\circ_\T b)=_{\A}r^- a \circ_\S r^-b \\
  \mbox{commutativity + localization +} & &  \\ \mbox{$\circ$ determined by singletons}
   & & \\
  \hline
\textbf{unital convergent cover} &
\mbox{unital comm. quantale} & \hfill\mbox{unital conv. cover map =} \\
 \S=(S,\cov,\circ,I) & (Sat(\A),\bigvee^\A,\circ^\A,\A I) & \mbox{convergent cover map } + \\
 \mbox{convergent cover + }a =_\A a\circ I & & r^-I_\T=_\A I_\S \\
   \hline \textbf{formal cover} &
\mbox{locale (frame)} & \mbox{formal cover map =} \\
 \mbox{convergent cover } +  & (Sat(\A),\bigvee^\A,\wedge^\A) & \textrm{ convergent cover map +}\\
  \circ^\A=\wedge^\A \mbox{ (weakening and contraction)} &  &  r^-T=_\A S \\
\hline
\end{array}
$$

\section{Connection with other definitions in the literature}
\label{section other definitions}

In this section we review the most relevant different presentations of
formal cover (derived from the corresponding versions of formal
topology) given in the literature. We prove constructively that they
all give rise to equivalent categories; this seems to appear here
explicitly for the first time. Moreover, we show that they can 
all be obtained as particular instances of our present definition.

\subsection*{The preorder induced by a basic cover}
\label{fregiu} For $(S,\cov)$ a basic cover and $a,b\in S$, we
consider the preorder (that is, the reflexive and transitive binary
relation)
\begin{equation}\label{def. <= indotto}
 a\leq b \quad\stackrel{def}
{\Longleftrightarrow}\quad a\cov b.
\end{equation}
As usual, we write $\fregiu U$ for the subset $\{a\in S$ $:$
$(\exists u\releps U)(a\leq u)\}$. Moreover, we put
 \begin{equation}\label{eq. def. fregiu indotto}U\fregiubi
V\, =\, (\fregiu U)\cap(\fregiu V)\ .\end{equation}
 Trivially, we have $a\fregiubi b$
$=$ $\A a\cap\A b$ and $U\fregiubi V$ $=$ $\bigcup_{a\releps U,
b\releps V}a\fregiubi b$ so that $\fregiu$ 
 is determined by its restriction on singletons. The map $U$ $\mapsto$ $\fregiu U$ is a
saturation (or closure) operator on $\P(S)$. Thus it makes sense to
consider the structure
$(Sat(\fregiu),\bigvee^{\fregiu},\wedge^{\fregiu})$ which is always
a lattice. Moreover, all the following hold:
\begin{enumerate}
\item $\fregiu$ is commutative \hfill $U\fregiubi V\ =\ V\fregiubi U$
\item $\fregiu$ is associative \hfill $(U\fregiubi V)\fregiubi W\ =\
U\fregiubi(V\fregiubi W)$
\item $\fregiu$ distributes over unions \hfill $U\fregiubi(\bigcup_{i\in I}V_i)\ =\ \bigcup_{i\in I}(U \fregiubi
V_i)$
\item $\fregiu$-left \hfill $\A(U\fregiubi V)\ \sub\ \A U\cap\A V$
\item contraction \hfill $U\fregiubi U\ =_\A\ U$
\end{enumerate} for every $U,V,U',V'\sub S$ and every set-indexed family
$\{V_i\}_{i\in I}$ in $\P(S)$. Note, however, that the operation
$\fregiu$ does not in general satisfy stability (it does not respect
$=_\A$), so the triple $(S,\cov,\fregiu)$  is
neither a formal cover, nor a convergent cover, nor even a basic
cover with operation (definition \ref{def. BCovwith operation}).

\subsection{$\cov$-formal covers}\label{formaltopologies}

Thanks to the discussion above and by proposition \ref{prop.
charact. formal cover}, it is immediate to see that:

\begin{lemma}
Given a basic cover $(S,\cov)$, the structure $(S,\cov,\fregiu)$ is a 
formal cover if and only if $\fregiu$ satisfies
$\fregiu$-right (in the sense of lemma \ref{equivcircmeet}).
\end{lemma}
This justifies the following (see \cite{cssv}):
\begin{definition}
Let $\S$ = $(S,\cov)$ be a basic cover. We say that $\S$ is a
\emph{$\cov$-formal cover} if $(S,\cov,\fregiu)$ is a formal cover,
that is, if \enspace $a\cov U\;\land \; a\cov V\;\Rightarrow \; a\cov(U\fregiu V)$ \enspace
for all $a\in S$ and all $U,V\sub S$.
\end{definition}

Clearly, $\cov$-formal covers can be identified with those
particular formal covers $(S,\cov,\circ)$ for which $\circ$ =
$\fregiu$, that is $a\circ b$ = $\{c\in S$ $|$ $c\cov a$ $\land$
$c\cov b\}$ for all $a,b\in S$. Thus:

\begin{definition} We call $\cov${\bf-FCov}  the full subcategory of {\bf FCov}
whose objects are $\cov$-formal covers.
\end{definition}

So a morphism $r$ between two $\cov$-formal covers $(S,\cov)$ and
$(S',\cov')$ is a morphism between the corresponding basic
covers which satisfies the extra conditions: $r^-(a\fregiubi' b)$
$=_\A$ $r^-a\fregiubi r^-b$ and $r^-S'$ $=_\A$ $S$.

\begin{lemma}\label{lemma circ eq. fregiu}
If $(S,\cov,\circ)$ is a formal cover, then $\circ$ coincides with
$\fregiu$ modulo $=_{\A}$, that is $U \circ V$ $=_{\A}$ $U \fregiubi
V$ for all $U,V \sub S$. Hence $(S,\cov)$ is a $\cov$-formal cover
and $(S,\cov,\circ)$ is isomorphic to $(S,\cov,\fregiu)$.
\end{lemma}
\begin{proof} To prove the first part of the statement, it is
sufficient to check that $a\fregiubi b$ $=_\A$ $a\circ b$ for all
$a,b\in S$. One gets $a\fregiubi b$ $\cov$ $a\circ b$ by
$\circ$-right and $a\circ b$ $\sub$ $a\fregiubi b$ by weakening.
The second part follows from lemma \ref{lemma uniqueness of circ for
formal covers}.
\end{proof}

This lemma gives immediately that:
\begin{corollary}
The categories $\cov${\bf-FCov} and {\bf FCov} are equivalent.
\end{corollary}

Finally,  one can show that  the inclusion functor from {\bf FCov} to {\bf BCov}
reflects isomorphisms:
\begin{proposition}
\label{fcisoasbciso}
Assume that $r:\S \to \S'$ is an isomorphism in {\bf BCov} with inverse $s$.  
Then $\S$ is a formal cover if and only if $\S'$ is a formal cover.
In this case, $r$ and $s$ form an isomorphism also in  {\bf FCov}.
\end{proposition}
\begin{proof}
The assumption means that the maps $\A r^-$ and $\A' s^-$ form a suplattice isomorphism between
$Sat(\A)$ and $Sat(\A')$, and hence in particular an order-isomorphism. 
By a general fact of
order theory,   both $\A r^-$ and $\A' s^-$ preserve  meets, besides joins;
 in fact, both $\A r^-$ is left adjoint to $\A' s^-$ and vice versa.
 So $Sat(\A)$ satisfies distributivity iff so does $Sat(\A')$, that is, $\S$ is a formal cover iff so is $\S'$.
 When $\S$ and $\S'$ are formal covers, meets are presented by $\fregiu$ and $\fregiu'$, that is, 
 $\A U \cap \A V = \A (U \fregiubi V)$ and  $\A' W \cap \A' Z = \A' (W \fregiubi Z)$.
 It is immediate to see that $\A r^-$ preserves meets means precisely that $r$ is convergent. Similarly for $s$. 
\end{proof}

\subsection{$\leq$-formal covers}

The approach via $\cov$-formal covers can appear as the most general
since it describes the meet of $Sat(\A)$ in terms of $\fregiu$ and
so, in the end, by means of the cover itself. However, it has a serious
drawback: to generate $\cov$ inductively one cannot use  axioms or
rules  involving $\fregiu$, at least constructively. In fact, the
operation $\fregiu$ is not well-defined until the process of
generation of $\cov$ is completed. Thus one has to modify the
presentation. One possibility is to approximate the meet by means of
a primitive operation $\circ$, as in the present paper, and only at
the end of the generation process one sees that $\circ$ and
$\fregiu$ coincide (up to $=_\A$).

Another way to modify the definition of $\cov$-formal cover in order to
include inductively generated examples is to use the definition in
\cite{ftopposets,cssv}. The idea is to start from a preordered base and hence
to define  $\fregiu$ depending on the preorder rather than on $\cov$. The
resulting notion is here called a $\leq$\emph{-formal cover}.
The inductive generation of $\leq$-formal covers corresponds to that
in \cite{Vickers-Johnstone} via {\it generators} (represented by the
preordered set of basic opens) and {\it relations} (the starting
axiom-set). This approach turned out to be crucial to describe
algebraic domains as \emph{unary} $\leq$-formal covers (see
\cite{sambin-domains as formal topologies}).

\begin{definition} A \emph{$\leq$-formal cover}
 is a basic cover whose carrier $S$ is equipped with a
preorder $\leq$ such that:
$$
\infer[\leq\textrm{-left}\qquad\textrm{and}]{a\cov U}{a\leq b & b\cov
U}\qquad\infer[\leq\textrm{-right}]{a\cov U\fregiubi_{\leq} V}{a\cov
U & a\cov V}
$$ 
where: $U\fregiubi_{\leq} V$ $=$
$\bigcup_{u\releps U,\, v\releps V}(u\fregiubi_{\leq} v)$ and
$u\fregiubi_{\leq} v$ $=$ $\{a\in S$ $|$ $a\leq u$ $\land$ $a\leq
v\}$.
\end{definition}

Clearly, for every $\leq$-formal cover, the structure
$(S,\cov,\fregiu_\leq)$  is a formal cover. Actually, the original   $\leq$-formal cover 
can be identified with $(S,\cov,\fregiu_\leq)$.
\begin{definition}
Let $\leq${\bf-FCov} be the full subcategory of {\bf FCov} whose
objects are $\leq$-formal covers.
\end{definition}
Every $\cov$-formal cover $(S,\cov,\fregiu)$ is clearly also a 
 $\leq$-formal cover, with $\leq$ defined as in (\ref{def. <= indotto}).
Conversely,  
every $\leq$-formal cover $(S,\cov,\fregiu_\leq)$ is isomorphic to the
$\cov$-formal cover $(S,\cov,\fregiu)$ by lemma~\ref{lemma circ eq. fregiu}.
 Therefore we can conclude that:
\begin{proposition}
The category $\leq${\bf-FCov} is equivalent to $\cov${\bf-FCov}
and hence to {\bf FCov}.
\end{proposition}

Note that the notion of $\leq$-formal cover \emph{does not admit a
generalization to quantales} because convergence expressed via the
operation $\fregiu_\leq$ already enjoys weakening and contraction.

\subsection{Formal covers with a monoid operation}

In the original definition~\cite{IFS} of  formal topology,  later revised in \cite{somepoints}, convergence
was defined by means of a primitive binary operation between basic
opens. 
This approach to pointfree topology turned out to be crucial in
order to represent in a predicative way  formal topologies on
algebraic structures \cite{peter}, Stone spaces~\cite{finitary
topologies,saranegri}, Scott domains~\cite{SVV} and exponentiations
between Scott domains~\cite{esc}. In the terminology of the present
paper, the notion of formal cover in~\cite{somepoints} can be rephrased as:

\begin{definition}
A \emph{$\bullet$-formal cover} is a formal cover $(S,\cov,\circ)$
together with a binary operation $\bullet\ :\ S\times S$
$\rightarrow$ $S$ on elements of $S$ such that $a\circ b$ $=$
$\{a\bullet b\}$ for all $a,b\in S$.
\end{definition}

\begin{definition}
We call $\bullet${\bf-FCov}  the full subcategory of {\bf FCov} whose
objects are $\bullet$-formal covers.
\end{definition}

We are now going to prove that the notion of $\bullet$-formal cover
is equivalent to that of $\cov$-formal cover (and hence to all the
other notions). We need the following

\begin{lemma}\label{lemma dottification}
For every basic cover $\S$ = $(S,\cov)$, there exists a basic cover
$Dot(\S)$ such that:
\begin{enumerate}
\item $Dot(\S)$ is isomorphic to $\S$ in {\bf BCov};
\item the carrier of $Dot(\S)$ is naturally endowed with a binary
operation $\bullet$;
\item $\S$ is a $\cov$-formal cover if and only if $Dot(\S)$
is a $\bullet$-formal cover; in that case they are isomorphic as
$\cov$-formal covers.
\end{enumerate}
\end{lemma}
\begin{proof}
Let $List (S)$ be the set of all finite lists of elements of $S$. As
usual $[\,]$ is the empty list and $[a_1,\ldots,a_n]$ is the list
whose elements are $a_1,\ldots,a_n\in S$. Let $\bullet$ denote
concatenation between lists; we extend $\bullet$ also to subsets  $K,L\subseteq List(S)$ in the
following way:
$$
K\bullet L=\{k\bullet l\ |\ k\releps K,\ l\releps L\}.
$$
Let $r$ be the relation between $S$
and $List(S)$ defined by:
$$
b\, r\,[a_1,\ldots,a_n]\ \Leftrightarrow\ (b\cov a_1)\land\cdots \land (b\cov a_n)
\qquad\textrm{and}\quad b\,r\,[\,]\quad\textrm{true}
$$ 
for all
$a_1,\ldots,a_n\in S$. In other words, we have $r^- [a_1,\ldots
,a_n]$ = $\A a_1\cap\ldots\cap\A a_n$ = $a_1\fregiubi\ldots\fregiubi
a_n$ and $r^- [\,]$ = $S$. Finally, let $l\cov' K$ be $r^-l\cov r^-K$, for all $l\in
List(S)$ and $K\sub List(S)$. We put $Dot(\S)$ =
$(List(S),\cov',\bullet)$. This completes the definition of
$Dot(\S)$ (and proves $2$). 

It is easy to check that $\cov'$ is a basic cover. By the very definition 
of $\cov'$, the relation $r$ becomes a morphism in
{\bf BCov} from $(S,\cov)$ to $(List(S),\cov')$. We check that this
is an isomorphism by defining its inverse. Consider the relation
$r'$ defined by $[a_1,\ldots,a_n]\, r'\,b$ if
$a_1\fregiubi\ldots\fregiubi a_n\cov b$ and by $[\,]\,r'\,b$ if 
$S\cov b$. In other words, for $l\in List(S)$ and $b\in S$, one has
$l\,r'\,b$ iff $r^-l\cov r^-[b]$ iff $l\cov'[b]$. So
$(r')^-b=\A'[b]=_{\A'}[b]$. Since $r$ is a basic cover map, it
follows that $r^-(r')^-b=_\A r^-[b]=_\A b$ and hence $r^-(r')^-U=_\A
U$ for all $U\sub S$. As a consequence, $r'$ is a basic cover map
from $(List(S),\cov')$ to $(S,\cov)$; in fact, if $a\cov U$, then
$r^-(r')^-a\cov r^-(r')^-U$, that is, $(r')^-a\cov'(r')^-U$ by
definition of $\cov'$. The equation $r^-(r')^-b=_\A b$ also shows
that $r'r=_\mathbf{BCov}id_{\S}$. It remains to be checked that
$rr'=_\mathbf{BCov}id_{Dot(\S)}$; for every list $l$,
$r^-(r')^-r^-l=_\A r^-l$ because $r'r=_\mathbf{BCov}id_{\S}$; so
$(r')^-r^-l=_{\A'} l$ by definition of $\A'$.
Summing up, $r$ is an isomorphism (with inverse $r'$) of basic
covers. This completes the proof of item 1. 

Because of proposition~\ref{fcisoasbciso}, to obtain 3 it is sufficient to show that
$Dot(\S)$ is a $\bullet$-formal cover iff $(List(S), \cov')$
is a $\cov$-formal cover.
One direction follows from lemma~\ref{lemma circ eq. fregiu}.
Conversely, first note that $r^-(k\bullet l)=
(r^-k)\fregiubi(r^-l)$ so that
  $\bullet$ satisfies weakening, associativity and commutativity.
Hence to prove that
 $Dot(\S)$ is a $\bullet$-formal cover it is sufficient to show that
 $\bullet$-right holds. To this aim, since $\fregiu'$-right holds, it is sufficient to show that
 $k \fregiubi' l \cov k \bullet l$.
So let $m\releps k\fregiubi' l$, that is, $m\cov'k$
and $m\cov'l$. This means $r^-m\cov r^-k$ and $r^-m\cov r^-l$. So
$r^-m\cov(r^-k)\fregiubi(r^-l)$ since  $\S$ satisfies
$\fregiu$-right by proposition~\ref{fcisoasbciso}. This is precisely $r^-m\cov r^-(k\bullet l)$, that
is, $m\cov' k\bullet l$. Hence $k\fregiubi' l$ $\cov'$ $k\bullet l$.
\end{proof}

\begin{corollary}
The category $\bullet${\bf -FCov} is equivalent to $\cov${\bf -FCov} and hence
also to {\bf FCov} and $\leq${\bf -FCov}.
\end{corollary}

By unfolding the equivalence between $\bullet${\bf -FCov} and
$\leq${\bf -FCov} one can deduce that a $\bullet$-formal cover
$(S,\cov,\bullet)$ is identified with the $\leq$-formal cover
$(S,\cov,\leq)$ where $a\leq b$ is $a\cov b$. However, in many cases
there exists a way to construct a $\leq$-formal cover
corresponding to a given $\bullet$-formal cover without using the
cover $\cov$ in the definition of $\leq$. For instance, if the
operation $\bullet$ on $S$ is associative (and not just associative
\emph{modulo} $=_\A$), then one can define a preorder on $S$ by:
$a\leq_m b$ if 
either $a=b$ or $a=l\bullet b$ or $a=b\bullet r$ or
$a=l\bullet b\bullet r$ for some $l,r\in S$.\footnote{In \cite{bs}
a similar
definition is given under the further assumption that
$\bullet$ is
commutative.}This is the smallest preorder making both $a\bullet b\leq_m a$ and
$a\bullet b\leq_m b$ true. In particular, $a\bullet b\releps
a\fregiubi_{\leq_m} b$ and hence $\leq$-right follows from the
corresponding property for $\bullet$. Moreover, $a\leq_m b$ yields
$a\cov b$; hence $\leq_m$-left holds and so also $a\fregiubi_{\leq_m} b\cov
a\bullet b$. Summing up, $(S,\cov,\leq_m)$ is a $\leq$-formal cover
and $a\fregiubi_{\leq_m} b=_\A a\bullet b$; hence $(S,\cov,\bullet)$ and
$(S,\cov,\leq_m)$ are isomorphic in {\bf FCov} by 
lemma~\ref{lemma uniqueness of circ for formal covers}.

We can also prove that the notions of formal cover presented here are 
essentially equivalent to the original notion in \cite{IFS}, where the base
is required to be a semilattice. There are several ways to see this. 
For instance,  given a $\cov$-formal cover $(S,\cov,\fregiu)$, one can modify 
the proof of lemma~\ref{lemma dottification} by taking $\P_\omega(S)$, 
the set of finite subsets of $S$ (see~\cite{finiteness}), instead of $List(S)$ 
and $\cup$ instead  of concatenation. By adapting the definition of the cover 
in the obvious way, one gets a formal cover $(\P_\omega(S),\cov',\cup)$ 
which is isomorphic to the given one and, moreover, whose base is a semilattice.

\subsection{Connection with other presentations of quantales}

In \cite{pretopologies} the third author introduced the notion of a
\emph{pretopology} in order to give constructive semantics for a
class of linear-like  logics. The same notion was used
in \cite{bs} to give a presentation of unital commutative quantales.

In our notation, a pretopology is essentially a unital convergent
cover such that $a\circ b$ is a singleton for all $a,b$. So
pretopologies form a category, say $\bullet${\bf-uCBCov}, which is
to {\bf uCBCov} as $\bullet${\bf-FCov} is to {\bf FCov}. By suitably
modifying lemma \ref{lemma dottification} (in particular, by replacing 
$\fregiu$ with $\circ$), it is possible to show
that $\bullet${\bf-uCBCov} and {\bf uCBCov} are equivalent.

\subsection{Remark on the unary and finitary cases}

When we pass to consider the unary or finitary  case, the connection
between  the different definitions  of formal
covers changes considerably.
Recall that a cover is \emph{finitary} if for all $a$ and $U$ with $a\cov U$, there exists
a finite subset $K$ of $U$ such that $a\cov K$.
A finitely cover is 
\emph{unary} if the subset $K$ in the definition has at most one element.

For instance, the equivalence between $\leq$-formal covers and
$\bullet$-formal covers no longer holds if we restrict to
their unary or finitary versions. Indeed, unary $\leq$-formal covers
are presentations of algebraic domains while unary $\bullet$-formal
covers present Scott domains (see \cite{sambin-domains as formal
topologies} and \cite{SVV}). The formal-topological presentation of
such classes of domains allows to see why Scott domains are closed
under exponentiation whilst algebraic domains are not. Indeed,
looking at the construction in \cite{MV04} of the exponential object
from an algebraic domain to an inductively generated formal cover
(in the category of inductively generated formal covers), one can
see that the exponentiation of two algebraic domains is not an
algebraic domain, in general.

Concerning the finitary case, we know that Stone spaces correspond
to finitary $\bullet$-formal covers (see \cite{saranegri} and
\cite{finitary topologies}). Similar characterizations of the
finitary versions of the other presentations are still unknown. In
particular, it is not clear what structures finitary $\cov$-formal covers
represent.

\section{Categorical reading of inductive generation}\label{catind}

In this section, we provide a categorical analysis of our modular
method for generating basic covers, convergent covers and formal
covers as in sections \ref{bcovi}, \ref{indqu} and \ref{inductive
generation formal covers}. We clarify in what sense the generation
of a formal cover or of a convergent cover from a given axiom-set is
{\it free}. This requires finding suitable categories in such a way
that our inductive generation processes provide object parts of 
right adjoints to suitable forgetful functors (right adjoints become
left adjoints, as expected, if one works in the opposite categories
following the direction of frame maps).

There exists a well known construction of the free frame over a
suplattice as a consequence of Johnstone's coverage theorem
\cite{stone spaces} (see also  theorem 4.4.2 in \cite{vickers}).
Here we consider this construction in the corresponding dual
categories, and hence we speak of the (co)free formal cover over an
inductively generated basic cover. We decompose it into three steps
as follows.
\begin{list}{-}{}
\item 
First, we use the base $S$ of the given basic cover to
construct a new inductively generated basic cover  $O(S)$, with different  
base and axiom-set,
naturally equipped with a pre-convergence operation on
subsets and a distinguished subset (that will
become the convergence operation and the unit of a quantale, respectively, in a later step).
\item
 Then, as explained in proposition \ref{prop. ind. gen. of
convergent covers}, we localize the  axiom-set of  $O(S)$ and generate a
unital convergent cover, that is a unital commutative quantale.
\item Finally, we add
the axioms of weakening and contraction (see proposition \ref{prop.
ind. gen. of formal covers}) thus obtaining a formal cover.
\end{list}

 These three steps give rise to three adjunctions that  are all in the form of a right adjoint to a forgetful
functor  $U$; the object part of each right adjoint is provided by one of
our methods of inductive generation.\footnote{A
similar decomposition holds also by taking the non commutative case
of $\circ$-basic covers and unital convergent covers, that is,
unital (not necessarily commutative) quantales.}
$$  
\def\objectstyle{\scriptstyle}
 \def\labelstyle{\scriptstyle}
 {\xymatrix@+1pc{
{\mbox{\bf BCov$_{i}$}}\ar@<-1.5ex>[r]_O^{\bot} & {\mbox{\bf
BCov$_{\circ,i}$}} \ar@<-1.5ex>[r]_{Q}^{\bot}\ar@<-1.5ex>[l]_{U} &
{\mbox{\bf uCBCov$_{i}$}}\ar@<-1.5ex>[l]_{U}
 \ar@<-1.5ex>[r]_{\quad L}^{\quad\bot} & \mbox{\bf FCov$_i$}\ar@<-1.5ex>[l]_{\quad U} }} 
 $$
The novelty of our decomposition lies in introducing a subcategory {\bf BCov}$_\circ$
of {\bf BCov}  whose objects include those
built in the first step of the above generation process.
The objects of  {\bf BCov}$_\circ$, called $\circ$-basic covers, are  basic
covers equipped with an operation on subsets of the
base that distributes over unions,  is associative, commutative and
 has a unit. This we call a pre-convergence operation; in fact, it enjoys 
 all  properties of a convergence operation (section~\ref{section convergent covers}) but localization.
  The morphisms of {\bf BCov}$_\circ$ are basic cover maps
preserving the pre-convergence operation and the units. 
$\circ$-basic covers represent the starting data
from which we can generate a formal or convergent cover.
 By using {\bf BCov}$_\circ$ we are able to recognize our inductively generated
formal covers and convergent covers as free structures. Indeed, in the
generation of formal covers, the category {\bf BCov}$_\circ$  plays a role 
analogous to that played by semilattices in Johnstone's coverage
theorem.

As we will see, the functor $Q$ is a categorical rendering of proposition \ref{prop.
ind. gen. of convergent covers} and lemma \ref{lemma continuity for
generated convergent covers}. A similar remark applies to the
functor L with respect to proposition \ref{prop. ind. gen. of formal
covers} and lemma \ref{lemma continuity for generated formal
covers}. The object part of the composition $L\cdot Q\cdot O$
coincides impredicatively with the construction of the free frame
over a suplattice. Impredicatively, the functor $Q\cdot L$ is
surjective on objects since \emph{every} locale is co-freely generated from a
$\circ$-basic cover. On the contrary, the functor $ L\cdot Q\cdot O$ is not
surjective on objects, since not every frame is free over
some suplattice.

Now we start by introducing the category of $\circ$-basic covers:
\begin{definition}
\label{bcirc} A $\circ$\emph{-basic cover} $\S=(S,\cov,\circ)$ is a
basic cover $(S,\cov)$ with an operation $\circ$ on subsets of $S$
such that
\begin{list}{-}{ }
\item  $\circ$ is distributive over unions  modulo $=_\A$, that is,
$\bigcup_{i\in I}(U_i\circ V)\ =_\A\ (\bigcup_{i\in I} U_i)\circ V$;
\item
$\circ$ is associative modulo $=_\A$, that is, $(U\circ V)\circ W\
=_\A\ U\circ(V\circ W)$;
\item $\circ$ is commutative modulo
$=_\A$, that is, $U\circ V\ =_\A\  V\circ U$;
\item there exists a unit, that is, a subset $I$ such that
$a\  =_\A \ a\circ I$
 for all $a\in S$.
\end{list}

Given two $\circ$-basic covers $\S$ and $\T$, a $\circ$\emph{-basic
cover map} from $\S$ to $\T$ is a basic cover map $r:(S,\cov_\S)\longrightarrow(T,\cov_\T)$ preserving 
operation and unit as follows:
$$r^-(U\circ_{\T} V)\ =_\A\ (r^-U)\circ_{\S}
(r^-V)\qquad\textrm{and}\qquad r^-{I_{\T}}\ =_\A I_{\S}$$ for all
$U,V\sub T$.
\end{definition}

It is easy to check that $\circ$-basic covers and $\circ$-basic
cover maps form a category. We say that a $\circ$-basic cover is
inductively generated when so is its underlying basic cover.

\begin{definition}
We call {\bf BCov$_{\circ}$} the category of $\circ$-basic covers
with $\circ$-basic cover maps. {\bf BCov$_{\circ,\, i}$} is the
full subcategory of {\bf BCov$_{\circ}$} whose objects are
inductively generated.
\end{definition}

These subcategories of basic covers are not relevant \emph{per se},
in the sense that they do not correspond to specific algebraic
structures. Their role is just that of explaining the universal
property of inductive generation for formal covers. 

Here we prove that there exists a functor $O$, from the
category of inductively generated basic covers to its subcategory of
$\circ$-basic covers, that is right adjoint to the corresponding forgetful
functor. This means that we can build the (co)free $\circ$-basic
cover generated from a basic cover:
\begin{proposition}\label{prop. adjunction1}
The forgetful functor from {\bf BCov$_{\circ,\,i}$} to {\bf
BCov$_i$}
 has a right adjoint
$$O:\ \mbox{\bf BCov$_{i}$}\, \longrightarrow \, \mbox{\bf BCov$_{\circ,\,i}$}\ .$$
\end{proposition}
\begin{proof}
Let $\S$ be a basic cover inductively generated by an axiom-set $I,C$. 
To define
the \emph{(co)free} $\circ$-basic cover over $\S$, called 
$O(\S)$, we start from the base $List(S)$. For $l,k\in List(S)$, let 
$l\circ_{O(\S)} k$ be (the singleton whose element is) $l\bullet k$, the
concatenation of $l$ and $k$. Let the unit $I_{O(\S)}$ be (the
singleton whose element is) $[\ ]$, the empty list. Then the cover
of $O(\S)$ is the least basic cover $\cov_{O(\S)}$
such that:
\begin{itemize}
\item[-] $[a]\cov_{O(\S)}\{[u]\ |\
u\releps C(a,i)\}$ holds for every $a\in S$ and $i\in I(a)$ (where $[b]$, for $b\in
S$, denotes the list of length one whose element is $b$);
\item[-] $l\bullet k\cov_{O(\S)} k \bullet l$ for all $l,k\in
List(S)$;
\end{itemize}
(associativity, as well as $l=_{\A_{O(\S)}} l\bullet[\ ]$, holds automatically).

We need to show that, for every inductively generated basic cover
$\S$, a morphism $i_\S$ $:$ $O(\S)\, \longrightarrow\, \S$
exists in {\bf BCov$_{i}$} such that, for every $r$ $:$
$\T\longrightarrow\S$ in {\bf BCov$_{i}$} with $\T$ an inductively
generated $\circ$-basic cover, there exists a unique $\circ$-basic
cover map $\widehat{r}$ such that the following diagram in {\bf
BCov}$_{i}$ commutes.
$$\xymatrix{ & & \S \\
\T \ar@{->}[urr]^{r} \ar@{->}[rr]_{\widehat{r}} & & O(\S)
\ar@{->}[u]_{i_\S} }$$ Let $i_\S$ be the relation between $List(S)$
and $S$ defined by ${i_\S}^-a$ = $\{[a]\}$. This is a basic cover
map
$$i_\S\ :\ (List(S),\cov_{O(\S)})\ \rightarrow\ (S,\cov_\S)$$
by lemma \ref{lemma continuity for generated covers}; indeed,
$i_{\S}^-a\cov_{O(\S)} {i_\S}^-C(a,i)$ holds by the
definitions of $\cov_{O(\S)}$ and $i_\S$. To complete the
proof it is sufficient to give the definition of $\widehat{r}$,
which is actually
 compulsory. In fact, $\widehat{r}^-[a]$ $=_{\A_\T}$ $r^-a$
because the diagram must commute, $\widehat{r}^-[\ ]$ $=_{\A_\T}$
$I_{\S}$ because $\widehat{r}$ must preserve units and finally
$\widehat{r}^-[a_1,\ldots,a_n]$ =
$\widehat{r}^-([a_1]\bullet\ldots\bullet[a_n])$ $=_{\T}$
$\widehat{r}^-[a_1]\circ\ldots\circ\widehat{r}^-[a_n]$ $=_{\T}$
$r^-a_1\circ\ldots\circ r^-a_n$ because $\widehat{r}$ must respect
convergence. It remains to be proved that $\widehat{r}$ is a basic
cover map. Since $O(\S)$ is inductively generated, it is
sufficient to check that $\widehat{r}^-[a]\cov_\T\widehat{r}^-\{[b]\
|\ b\releps C(a,i)\}$, that is, $r^-a\cov_\T r^-C(a,i)$ which holds
because $r$ is a basic cover map.
\end{proof}

The right adjoint in the previous proposition represents the first
step to build the co-free formal cover generated from a basic cover. The
second step is to add the axioms making $O(\S)$ a unital convergent cover 
(unital commutative quantale). Also this step enjoys a
universal property, namely there exists a functor $Q$, from 
\mbox{\bf BCov$_{\circ,\,i}$} to its
subcategory \mbox{\bf uCBCov$_{i}$} of inductively generated unital
convergent covers, that is right adjoint to the corresponding forgetful functor.

\begin{proposition}\label{prop. adjunction2}
The forgetful functor from \mbox{\bf uCBCov$_{i}$} to {\bf
BCov$_{\circ,\,i}$} has a right adjoint
$$Q:\ \mbox{\bf BCov$_{\circ,\,i}$}\, \longrightarrow \, \mbox{\bf uCBCov$_{i}$}$$
that is surjective on objects.
\end{proposition}
\begin{proof}
Let $\S$ be a $\circ$-basic cover inductively generated by an axiom-set $I,C$. We call
$Q(\S)$ the unital convergent cover generated from the same
axiom-set, the same $\circ$ and the same unit of $\S$ as described
in proposition~\ref{prop. ind. gen. of convergent covers} and the
remark following it. We need to show that for every
inductively generated $\circ$-basic cover $\S$ there exists a
morphism $j_\S$ $:$ $Q(\S)\, \longrightarrow\, \S$ in {\bf
BCov$_{\circ,\,i}$} such that, for every $r$ $:$
$\T\longrightarrow\S$ in {\bf BCov$_{\circ,\,i}$} with $\T$ a unital
convergent cover, there exists a unique unital convergent cover map
$\widetilde{r}$ such that the following diagram in {\bf
BCov}$_{\circ,\,i}$ commutes.
$$\xymatrix{ & & \S \\
\T \ar@{->}[urr]^{r} \ar@{->}[rr]_{\widetilde{r}} & & {Q}(\S)
\ar@{->}[u]_{j_\S} }$$ Let $j_\S$ be the identity relation on the
set $S$. This is a $\circ$-basic cover map from $Q(\S)$ to $\S$. In
fact $a\cov_{Q(\S)} C(a,i)$ holds for all $a\in S$ and $i\in I(a)$,
because $Q(\S)$ is generated by an axiom-set extending $I,C$.
Moreover,  $j_\S$ respects convergence and units since $Q(\S)$ has
the same $\circ$ and unit of $\S$. To complete the proof it is
sufficient to define $\widetilde{r}$ as $r$ itself given that $r$ is
also a map toward $Q(\S)$ by lemma~\ref{lemma continuity for
generated convergent covers}.

Since every inductively generated unital convergent cover can be
obtained as in proposition~\ref{prop. ind. gen. of convergent
covers} we conclude that $Q$ is surjective.
\end{proof}

The above proposition is a refinement of the universal
property of quantale presentations in \cite{bs}. In fact, the
authors of \cite{bs} work with a monoid operation on the base. Hence
their result corresponds to the existence of a right adjoint,
from  the subcategory {\bf
BCov$_{\bullet,\,i}$} (see below for a precise definition) of 
{\bf BCov$_{\circ,\,i}$},
 that is just a restriction of our functor $Q$.

The last step is the generation of a formal
cover from an inductively generated unital convergent cover.

\begin{proposition}\label{prop. adjunction3}
The forgetful functor from {\bf FCov$_{i}$} to \mbox{\bf
uCBCov$_{i}$} has a right adjoint
$$L:\ \mbox{\bf uCBCov$_{i}$}\, \longrightarrow \, \mbox{\bf FCov$_{i}$}$$
that is surjective on objects.
\end{proposition}
\begin{proof}
Let $\S$ be a convergent cover with unit inductively generated by an axiom-set $I,C$. We call
$L(\S)$ the formal cover generated from the same $I,C$ and the
same $\circ$ of $\S$ as in proposition~\ref{prop. ind. gen. of
formal covers}. We show that there exists a morphism
$k_\S$ $:$ $L(\S)\, \longrightarrow\, \S$ in {\bf uCBCov$_{i}$} such
that, for every $r:\T\longrightarrow\S$ in {\bf uCBCov$_{i}$} with
$\T$ a formal cover, there exists a unique continuous map
$\overline{r}$ such that the following diagram in {\bf uCBCov}$_{i}$
commutes.

$$\xymatrix{ & & \S \\
\T \ar@{->}[urr]^{r} \ar@{->}[rr]_{\overline{r}} & & {L}(\S)
\ar@{->}[u]_{k_\S} }$$ Let $k_\S$ be the identity relation on the
set $S$. This is a unital convergent cover map from $L(\S)$ to $\S$
by lemma~\ref{lemma continuity for generated convergent covers}. In
fact $a\cov_{L(\S)}
C(a,i)$ holds for all $a\in S$ and $i\in I(a)$, because $L(\S)$ is
generated by an axiom-set extending $I,C$. Moreover, $k_\S$
trivially respects $\circ$ and units. Finally, we
define $\overline{r}$ as $r$ itself, since $r$ is also a map
into $L(\S)$ by lemma~\ref{lemma continuity for generated formal
covers}.

Since every inductively generated formal cover can be obtained from
an inductively generated unital convergent cover by adding the
axioms of weakening and contraction, we conclude that $L$ is
surjective.
\end{proof}

The functors $O$, $Q$ and $L$   
give a decomposition of the right adjoint of the forgetful
functor from formal covers to basic covers:
\begin{corollary}\label{prop. composition}
The functor $L\cdot Q \cdot O:\  \mbox{\bf BCov$_{i}$}\
\longrightarrow \ \mbox{\bf FCov$_i$} $ is a right adjoint to the
forgetful functor from {\bf FCov$_i$} to {\bf BCov$_{i}$}.
\end{corollary}

This is a predicative counterpart of the existence of a left adjoint
to the forgetful functor from the category of frames to that of
suplattices.
If $\S$  already is  a formal cover, that is $Sat(\A)$ is a locale, then the construction
$L\cdot Q\cdot O\,(\S)$  gives a predicative presentation of
what is known as the \emph{lower power locale} \cite{vickers} of
$Sat(\A)$.

>From propositions \ref{prop. adjunction2} and \ref{prop. adjunction3}, we conclude:
\begin{corollary}\label{freef}
The functor $L\cdot Q:\  \mbox{\bf BCov$_{\circ,i}$}\
\longrightarrow\ \mbox{\bf FCov$_i$}$ is right adjoint to the forgetful functor from
{\bf FCov$_i$} to {\bf BCov$_{\circ,i}$} and, moreover, is surjective on objects.
\end{corollary}

This result shows that the method of inductively generating formal
covers from an axiom-set enjoys a universal property with respect to
the category {\bf BCov$_{\circ,i}$}.
Besides the functor $L\cdot Q$ presented here, there are in the
literature other ways for generating a formal cover given some
initial data. We need to recall three such methods to be able to
explain how our approach is a refinement of them all.

In \cite{cssv} it is shown how to generate a $\leq$-formal cover
starting from an axiom-set on a preordered set $(S,\leq)$. One can
easily check that, in the specific case in which $\delta(a\circ b)$
= $a\fregiubi_\leq b$, our rules to generate formal covers
(proposition \ref{prop. ind. gen. of formal covers}) become
perfectly equivalent to those given in \cite{cssv}. We now can see
that the construction in \cite{cssv} is {\it co-free} with respect to a suitable
subcategory of {\bf BCov}$_{\circ,\,i}$.

\begin{definition} We call {\bf
BCov}$_\leq$ the full subcategory of {\bf BCov}$_\circ$ whose
objects satisfy $\circ = \fregiu_\leq$ for some preorder $\leq$ on
the base. {\bf BCov}$_{\leq,i}$ is its full subcategory whose
objects are inductively generated.
\end{definition}

As a consequence of the results above, we have:

\begin{corollary}
The functor $L\cdot Q$ restricts to a functor from {\bf
BCov}$_{\leq,i}$ to the category $\leq$-{\bf FCov}$_i$ of inductively generated
$\leq$-formal covers. This restriction is right adjoint to the  forgetful
functor in the opposite direction and, moreover, is surjective on objects.
\end{corollary}

This adjunction explains in what sense the construction 
of inductively generated $\leq$-formal covers in \cite{cssv} is co-free. 
A direct proof can be obtained  by the instantiation of 
lemma~\ref{lemma continuity for generated formal covers} to formal
covers where $\circ=\fregiu_{\leq}$.
This  gives precisely the result which has played  a key role when 
dealing with inductively generated $\leq$-formal
covers (see for instance \cite{milly2005}). 
Note, however, that the
adjunction between {\bf BCov}$_{\leq,i}$ and $\leq$-{\bf FCov}$_i$
cannot be decomposed via quantales. In fact, the functor $Q$, when
applied to an object in {\bf BCov}$_{\leq,i}$, gives already a
$\leq$-formal cover (weakening and contraction automatically hold).
In other words, $L\cdot Q$ and $Q$ coincide on {\bf
BCov}$_{\leq,i}$.

The second approach we want to recall is that described in
\cite{bs}. There a monoid structure is assumed on the base. So the
starting data for generating a formal cover can be thought of as
objects of the following category.

\begin{definition} We call {\bf
BCov}$_\bullet$ the full subcategory of {\bf BCov}$_\circ$ whose
object satisfy $a\circ b=\{a\bullet b\}$ (for all $a,b\in S$) for
some monoid operation $\bullet$ on the base $S$. {\bf
BCov}$_{\bullet,i}$ is its full subcategory whose objects are
inductively generated.
\end{definition}

As before, the functor $L\cdot Q$ can be restricted to {\bf
BCov}$_{\bullet,i}$ to obtain the following categorical result
of the universal property in \cite{bs}:

\begin{corollary}
The functor $Q$  restricted to a functor from {\bf
BCov}$_{\bullet,i}$ to $\bullet${\bf-uCBCov}  is a right adjoint
to the  forgetful functor in the opposite direction.
\end{corollary}

We finally come to the third approach, namely Johnstone's coverage
theorem in \cite{stone spaces,vickers}. As one can recognize, Johnstone's result
reads in our framework just as proposition \ref{prop. ind. gen. of formal covers} in
the case in which $S$ has a $\wedge$-semilattice structure and
$a\circ b=\{a\wedge b\}$. Hence the data required
by Johnstone's method amount to an inductively generated basic cover
whose base has a semilattice structure. So we give the following:

\begin{definition} We call {\bf
BCov}$_\wedge$ the full subcategory of {\bf BCov}$_\circ$ whose
objects satisfy $a\circ b=\{a\wedge b\}$ (for all $a,b\in S$) for
some semilattice operation $\wedge$ on the base $S$. {\bf
BCov}$_{\wedge,i}$ is its full subcategory whose objects are
inductively generated.
\end{definition} Clearly {\bf BCov}$_\wedge$ is a subcategory both
of {\bf BCov}$_\leq$ and of {\bf BCov}$_\bullet$. With this
notation, Johnstone's result becomes:

\begin{proposition}
The functor $L\cdot Q$ restricted to a functor from {\bf
BCov}$_{\wedge,i}$ to {\bf FCov}$_i$   is right adjoint to the
forgetful functor in the opposite direction.
\end{proposition}

The following diagram summarizes the adjunctions discussed above
(each forgetful functor $U$ is a left adjoint) together with the
(full) inclusions between the various categories. It is easy to check
that all sub-diagrams commute.
$$\xymatrix{
&& {\bf BCov}_{\wedge,i} \ar@{^{(}->}[d] \ar@<1ex>@/^2pc/[ddrr]
  \ar@{^{(}->}@/_1.5pc/@<-4ex>[ddd] \\
&& {\bf BCov}_{\leq,i}\ar@{^{(}->}[d] \ar@/^1pc/[drr] \\
{\bf BCov}_i \ar@<-0.5ex>[rr]_{O} && {\bf BCov}_{\circ,i}
\ar@<-0.5ex>[ll]_{U} \ar@<-0.5ex>[r]_{Q} & {\bf uCBCov}_i
\ar@<-0.5ex>[r]_{L} \ar@<-0.5ex>[l]_{U} & {\bf FCov}_i
\ar@<-0.5ex>[l]_{U} \ar@<-1ex>@/_1pc/[ull]_{U}
\ar@<-2ex>@/_2pc/[uull]_{U} \ar@/^1pc/[dl]_{U} \\
&& {\bf BCov}_{\bullet,i} \ar@{^{(}->}[u] \ar@<-0.5ex>[r] &
\bullet-{\bf uCBCov}_i \ar@<-0.5ex>[l]_{U} \ar@{^{(}->}[u]
 \ar@<-1ex>@/_1pc/[ur]
 }$$

There is also another way to obtain a formal cover from an arbitrary
basic cover, which however does not lead to a functor from {\bf
BCov}$_i$ to {\bf FCov}$_i$. This construction starts from a basic
cover $\S=(S,\cov)$ generated from an axiom-set $I,C$ and applies
the method of proposition \ref{prop. ind. gen. of formal covers}
with $\delta(a,b)=a\fregiu b$, where $\fregiu$ is defined through
$\cov$ itself as in (\ref{def. <= indotto}) and (\ref{eq. def.
fregiu indotto}). No doubt, this method produces a formal cover,
namely $L\cdot Q\,(S,\cov,\fregiu)$. However, it cannot be extended
to a functor from {\bf
BCov}$_i$ to {\bf FCov}$_i$ since there is no reason for a basic cover map to
respect the operation $\circ=\fregiu$. Note also that the formal
cover $L\cdot Q\,(S,\cov,\fregiu)$ is quite different from $L\cdot
Q\cdot O\,(S,\cov)$. This is well visible when $\S$ itself is
a $\cov$-formal cover. In fact, in this case $L\cdot
Q\,(S,\cov,\fregiu)$ is $(S,\cov,\fregiu)$ itself, while $L\cdot
Q\cdot O\,(\S)$  presents the lower power locale of
$Sat(\A)$ and hence it is not isomorphic to $\S$. In general, one
can see that the formal cover $L\cdot Q\,(S,\cov,\fregiu)$ presents
\emph{the largest frame contained in the suplattice} $Sat(\A)$.

\section*{Conclusions}

We have presented a new definition of \emph{formal cover/formal topology} 
that generalizes all other definitions given so far. This has been
obtained by considering an operation $\circ$ between subsets which
is uniquely determined by its trace on singletons. Our approach
seems to gather all the advantages of previous definitions. It
allows us to reach the definition of formal cover in a modular way
passing through the case of quantales, as it happens for the
approach via a monoid operation on the base (see $\bullet$-formal
covers in section \ref{section other definitions}). At the same
time, it provides a uniform method of inductive generation that
includes the one originally introduced in \cite{cssv}.

Our new definition of  convergent covers and formal covers  allows us
to reproduce  Joyal-Tierney's characterization of quantales
and locales \cite{galoisexten} in a straightforward way. Previous
definitions of formal cover were not apt to this purpose. The
definition of $\leq$-formal cover does not generalize to represent
quantales. Also the original definition~\cite{IFS} is too specific;
it can be generalized as in \cite{pretopologies} and \cite{bs} to
represent quantales, but the only way one can see to obtain
Joyal-Tierney's characterization is to pass through the equivalence
with the new notion introduced here. 

Our new presentation of
convergence offers a uniform and modular method for generating
formal covers, convergent covers and basic covers inductively. This
uniformity allows us to recognize in what sense the various
inductive constructions provide free structures and thus refine
Johnstone's coverage theorem. Our analysis supplies a neat
decomposition of the well known adjunction associated to the free
frame over a suplattice. This would hardly be possible with the
presentation of formal covers in the literature.

\subsection*{Acknowledgements}

We thank Pino Rosolini for fruitful discussions on
Joyal-Tierney's representation of frames as monoids and Steve Vickers for 
pointing out the connection between free
frames over suplattices and lower power locales.

\end{document}